\documentclass[11pt, a4paper]{article}

\usepackage{listings}
\usepackage{}

\usepackage{tabularx}
\usepackage{indentfirst}
\usepackage{lineno}
\usepackage{graphicx}
\usepackage{ae}
\usepackage{amsmath}
\usepackage{amssymb}
\usepackage{latexsym}
\usepackage{url}
\usepackage{epsfig}
\usepackage{cite}
\usepackage{mathrsfs}
\usepackage{amsfonts}
\usepackage{amsthm}
\usepackage{float}
\usepackage{multirow}
\usepackage[titletoc]{appendix}
\usepackage[section]{placeins}

\usepackage{subfigure}
\usepackage{caption}
\usepackage{array}
\usepackage{longtable}
\usepackage{supertabular}
\newcommand{\PreserveBackslash}[1]{\let\temp=\\#1\let\\=\temp}
\newcolumntype{C}[1]{>{\PreserveBackslash\centering}p{#1}}
\newcolumntype{R}[1]{>{\PreserveBackslash\raggedleft}p{#1}}
\newcolumntype{L}[1]{>{\PreserveBackslash\raggedright}p{#1}}

\def\qed{\hfill$\Box$\vspace{12pt}}

\usepackage{color}

\long\def\delete#1{}

\newcommand{\be}{\begin{equation}}
\newcommand{\ee}{\end{equation}}
\newcommand{\ben}{\begin{equation*}}
\newcommand{\een}{\end{equation*}}
\newcommand{\bea}{\begin{eqnarray}}
\newcommand{\eea}{\end{eqnarray}}
\newcommand{\bean}{\begin{eqnarray*}}
\newcommand{\eean}{\end{eqnarray*}}

\newtheorem{thm}{Theorem}[section]
\newtheorem{cor}[thm]{Corollary}

\newtheorem{lem}[thm]{Lemma}

\newtheorem{defn}[thm]{Definition}

\numberwithin{equation}{section}

\title{Distance matrix correlation spectrum of graphs
\thanks{Supported by the National Natural Science Foundation of China (No.11361033).}}

\author{Pengli Lu\thanks{Corresponding author.   E-mail
addresses: lupengli88@163.com (\textbf{P. Lu}), lwzvay123@163.com (\textbf{W. Liu}).} \;and\; Wenzhi Liu
\\
\footnotesize{School of Computer and Communication, Lanzhou University of Technology, Lanzhou, 730050, Gansu, P.R. China}}

\date{}

\begin{document}

\openup 0.5\jot
\maketitle

\begin{abstract}
 Let $G$ be a simple, connected graph, $\mathcal{D}(G)$ be the distance matrix of $G$, and $Tr(G)$ be the diagonal matrix of vertex transmissions of $G$. The distance Laplacian matrix and distance signless Laplacian matrix of $G$ are defined by $\mathcal{L}(G)=Tr(G)-\mathcal{D}(G)$ and $\mathcal{Q}(G)=Tr(G)+\mathcal{D}(G)$, respectively. The eigenvalues of $\mathcal{D}(G)$, $\mathcal{L}(G)$ and $\mathcal{Q}(G)$ is called the $\mathcal{D}-$spectrum, $\mathcal{L}-$spectrum and $\mathcal{Q}-$spectrum, respectively. The generalized distance matrix of $G$ is defined as $\mathcal{D}_{\alpha}(G)=\alpha Tr(G)+(1-\alpha)\mathcal{D}(G),~0\leq\alpha\leq1$, and the generalized distance
spectral radius of $G$ is the largest eigenvalue of $\mathcal{D}_{\alpha}(G)$. In this paper, we give a complete description of the $\mathcal{D}-$spectrum, $\mathcal{L}-$spectrum and $\mathcal{Q}-$spectrum of some graphs obtained by operations. In addition, we present some new upper and lower bounds on the generalized distance spectral radius of $G$ and of its line graph $L(G)$, based on other graph-theoretic parameters, and characterize the extremal graphs. Finally, we study the generalized distance spectrum of some composite graphs.
    \bigskip


\end{abstract}

\section{Introduction}
All graphs considered are finite, simple and connected. Let $G=(V(G),E(G))$ be a graph with vertex set $V(G)$ and edge set $E(G)$. Let $d_{i}$ be the degree
of the vertex $v_{i}$ in $G$ for $i=1,2,\cdots,n$ and satisfy $d_{1}\geq d_{2}\geq\cdots\geq d_{n}$. Let $A(G)=(a_{ij})_{n\times n}$
be the $(0,1)$-adjacency matrix of $G$, where $a_{ij}=1$ if $v_{i}$ and $v_{j}$ are adjacent and $0$ otherwise,
and $D(G)=diag(d_{1},d_{2},\cdots,d_{n})$ be the degree diagonal matrix. The spectrum of $A(G)$ is denoted by $\lambda_{1}(G)\geq \lambda_{2}(G) \geq \cdots\geq \lambda_{n}(G)$.

The distance between two vertices $u$ and $v$ in $G$, denoted by $d_{G}(u,v)$, is defined to be the length of the shortest path between $u$ and $v$. The distance matrix $\mathcal{D}(G)=(d_{uv})$ of $G$ is the matrix indexed by vertices of $G$ with $d_{uv}=d_{G}(u,v)$. The eigenvalues of $\mathcal{D}(G)$ are denoted by $\mu_{1}^{\mathcal{D}}(G)\geq \mu_{2}^{\mathcal{D}}(G) \geq \cdots\geq \mu_{n}^{\mathcal{D}}(G)$, the multiset of all eigenvalues of $\mathcal{D}(G)$ is called the distance spectrum of $G$.
The transmission $Tr(u)$ of a vertex $u$ in $G$ is defined to be the sum of the distances from $u$ to all other vertices in $G$, i.e., $Tr(u)=\sum_{v\in V(G)}d_{G}(u,v)$. A graph $G$ is said to be $k$-transmission regular if $Tr(u)=k$, for each $u\in V(G)$. Hence the transmission degree sequence is given by $\{Tr_{1},Tr_{2},\cdots,Tr_{n}\}$. The second transmission degree of $v_{i}$, denoted by $T_{i}$, is given by $T_{i}=\sum_{j=1}^{n}d_{ij}Tr_{j}$.

Similarly to the Laplacian matrix and signless Laplacian matrix of graphs, the distance Laplacian matrix and distance signless Laplacian matrix are introduced by M. Aouchiche and P. Hansen \cite{1}. Let $Tr(G)=diag(Tr_{1},Tr_{2},\cdots,Tr_{n})$ be the diagonal matrix of the vertex transmissions in $G$, then $\mathcal{L}(G) = Tr(G)-\mathcal{D}(G)$ and $\mathcal{Q}(G) = Tr(G)+\mathcal{D}(G)$ are called the distance Laplacian matrix and distance signless Laplacian matrix, respectively. The spectrum of $\mathcal{L}(G)$ and $\mathcal{Q}(G)$ are denoted by $\mu_{1}^{\mathcal{L}}(G)\geq \mu_{2}^{\mathcal{L}}(G) \geq \cdots\geq \mu_{n}^{\mathcal{L}}(G)$ and $\mu_{1}^{\mathcal{Q}}(G)\geq \mu_{2}^{\mathcal{Q}}(G) \geq \cdots\geq \mu_{n}^{\mathcal{Q}}(G)$, respectively.

The average transmission is denoted by $t(G)$ and is defined by $t(G)=\frac{1}{n}\sum_{i=1}^{n}Tr_{G}(v_{i})$, then the distance energy of a connected graph $G$ was defined in \cite{t9} as
\begin{equation*}
DE(G)=\sum_{i=1}^{n}|\mu_{i}^{\mathcal{D}}(G)|.
\end{equation*}
Its mathematical properties were extensively investigated, see the recent articles \cite{t9,t10,t11,t12} and the references cited
therein.

The distance Laplacian energy and the distance signless Laplacian energy of G are defined as
\begin{equation*}
DLE(G)=\sum_{i=1}^{n}|\mu_{i}^{\mathcal{L}}(G)-t(G)|   \qquad  and \qquad  DSLE(G)=\sum_{i=1}^{n}|\mu_{i}^{\mathcal{Q}}(G)-t(G)|
\end{equation*}
, respectively. The distance Laplacian energy of a graph $G$ was first defined in \cite{13}, where several lower and upper bounds were obtained.

 In \cite{2}, Nikiforov proposed to study the convex linear combinations of the adjacency matrix and diagonal degree matrix of $G$, which reduces to merging the adjacency spectral and signless Laplacian spectral theories. Similarly, In \cite{4}, Guixian Tian and Shuyu Cui studied the convex combinations $\mathcal{D}_{\alpha}(G)$ of $Tr(G)$ and $\mathcal{D}(G)$ defined by
 \begin{equation*}
\begin{aligned}  
\mathcal{D}_{\alpha}(G)=\alpha Tr(G)+(1-\alpha)\mathcal{D}(G),~~0\leq\alpha\leq1.
\end{aligned}
\end{equation*}
Obviously,
\begin{equation*}
\begin{aligned}  
\mathcal{D}_{0}(G)=\mathcal{D}(G),~~~\mathcal{D}_{\frac{1}{2}}(G)=\frac{1}{2}\mathcal{Q}(G),~~~\mathcal{D}_{1}(G)=Tr(G),
\end{aligned}
\end{equation*}
and
\begin{equation*}
\begin{aligned}  
\mathcal{D}_{\alpha}(G)-\mathcal{D}_{\beta}(G)=(\alpha-\beta)\mathcal{L}(G).
\end{aligned}
\end{equation*}
Let $G=(V(G),E(G))$ be a connected graph of order $n$. Then all eigenvalues of $\mathcal{D}_{\alpha}(G)$ are denoted by $\rho_{1}(G)\geq\rho_{2}(G)\geq\cdots\geq\rho_{n}(G)$. The multiset of all
eigenvalues of $\mathcal{D}_{\alpha}(G)$ is called the generalized distance spectrum of $G$, denoted by
$\{\rho_{1}(G),\rho_{2}(G),\cdots,\rho_{n}(G)\}$. In particular, the largest eigenvalues of
$\mathcal{D}_{\alpha}(G)$ are denoted by $\rho(G)$. If $G$ is connected, then $\mathcal{D}_{\alpha}(G)$ is symmetric, nonnegative and irreducible.  By the Perron-Frobenius theorem, $\rho(G)$ is positive and simple, and there is a unique positive unit eigenvector $X$ corresponding to $\rho(G)$, which is called the generalized distance Perron vector of $G$.

A column vector $X=(x_{1},x_{2},\cdots,x_{n})^{T}\in R^{n}$ can be considered as a function defined on $V(G)$ which maps vertex $v_{i}$ to $x_{i}$, that is, $X(v _{i})=x_{i}$ for $i=1,2,\cdots,n$. Then,
\begin{equation*}
\begin{aligned}  
X^{T}\mathcal{D}_{\alpha}(G)X
\end{aligned}=\alpha\sum_{i=1}^{n}Tr(v_{i})x_{i}^{2}+2(1-\alpha)\sum_{1\leq i<j\leq n}d(v_{i},v_{j})x_{i}x_{j},
\end{equation*}
and $\rho$ is an eigenvalue of $\mathcal{D}_{\alpha}(G)$ corresponding to the eigenvector $X$ if and only if
$X\neq0$ and for each $i\in V(G)$,
\begin{equation*}
\begin{aligned}  
\rho x_{i}=\sum_{k=1}^{n}d_{ik}((1-\alpha)x_{k}+\alpha x_{i}).
\end{aligned}
\end{equation*}
These equations are called the $(\rho,x)$-eigenequations of $G$. For a normalized column
vector $X\in R_{n}$ with at least one nonnegative component, by the Rayleigh's principle,
we have
\begin{equation*}
\begin{aligned}  
\rho(G)\geq X^{T}\mathcal{D}_{\alpha}(G)X,
\end{aligned}
\end{equation*}
with equality if and only if $X$ is the generalized distance Perron vector of $G$.

Up till now, the distance spectrum of a connected graph has been investigated extensively, see the recent survey \cite{5} as well as the references therein. Recently, the distance Laplacian spectrum and distance signless Laplacian spectrum of graphs have also been studied in many papers. For example, Aouchiche and Hansen \cite{6} showed that the distance Laplacian eigenvalues and distance signless Laplacian eigenvalues do not decrease when an edge is deleted. In \cite{7}, the same authors proved that the star is the unique tree with minimum distance Laplacian spectral radius. In \cite{8}, Alhevaz et al. gave some upper and lower bounds on distance signless Laplacian spectral radius and also determined the distance signless Laplacian spectrum of some graph operations. In \cite{4}, Guixian Tian et al. defined the generalized distance matrix and studied Some spectral properties. Furthermore, they obtained some upper and lower bounds of spectral radius of the generalized distance matrix. Finally, the generalized distance spectra of some graphs obtained by operations are also studied.
For more review about distance Laplacian spectrum and distance signless Laplacian spectrum of graphs, readers may refer to \cite{9,10,11,111,112} and the references therein.

The paper is organized as follows. In Section 2, we give a list of some previously known results and the definition of some graph operation. In Section 3, we obtain the $\mathcal{D}-$spectrum, $\mathcal{L}-$spectrum and $\mathcal{Q}-$spectrum of the cluster of a distance regular graph with complete graph, the double graph of $G$, the join of regular graphs, the join of a regular graph with the union of regular graphs, the subdivision$-$edge join, the subdivision$-$vertex join and the subdivision$-$(vertex$-$edge) join of two regular graphs. These results enable us to study the distance Laplacian energy and the distance signless Laplacian energyof some special graphs.
In Section 4, we obtain some upper and lower bounds on the generalized distance spectral radius and
determine the extremal graphs in terms of transmission degree $Tr_{i}$, second transmission degree $T_{i}$, maximum degree $\bigtriangleup_{1}$, second maximum degree $\bigtriangleup_{2}$, minimum degree $\delta_{1}$, second minimum degree $\delta_{2}$ and so on. Further, we study the line graphs $L(G)$ of simple
connected graphs and determine some lower bounds on the generalized distance spectral radius of $L(G)$ based on some graph invariants, and characterize
the extremal graphs. In Section 5, we focus mainly on some graph operations and determine the generalized distance spectrum of some graphs obtained
by these operations.
\section{Preliminaries}
In this section, we shall list some previously known results that will be needed in the proofs of our results in the next three sections.
\begin{defn}\cite{12}\label{lem0}
If graphs $G_{1}$ and $G_{2}$ have no common vertices, It is said to disjoint union $G_{1}\cup G_{2}$;
The join of two vertex disjoint graphs $G_{1}$ and $G_{2}$, denoted by $G_{1}\triangledown G_{2}$, is the graph by joining each vertex of $G_{1}$ with every vertex of $G_{2}$.
\end{defn}
\begin{defn}\cite{13}\label{lem1}
Let $H$ be a graph rooted at $u$. Then given a graph G with vertex set $\{v_{1},v_{2}, \cdots , v_{p}\}$, the cluster $G\{H\}$ is defined as the graph obtained by taking $p$ copies of $H$ and for each $i$, joining the $i$th vertex of $G$ to the root in the $i$th copy of $H$.
\end{defn}
\begin{defn}\cite{14}\label{lem2}
 Let $G$ be a graph with vertex set $V(G)=\{v_{1}, v_{2},\cdots,v_{p}\}$. Take another copy of $G$ with the vertices labelled by $\{u_{1}, u_{2},\cdots, u_{p}\}$,  where $u_{i}$ corresponds to $v_{i}$ for each $i$ . Make $u_{i}$ adjacent to all the vertices in $N(v_{i})$ in G , for each $i$. The resulting graph, denoted by $D_{2}G$ , is called the double graph of $G$.
\end{defn}
\begin{defn}\cite{15}\label{lem3}
The subdivision$-$edge join of two vertex disjoint graphs $G_{1}$ and $G_{2}$, denoted by $G_{1}\oplus G_{2}$, is the graph obtained from $S(G_{1})$ and $S(G_{2})$ by joining each vertex of $I(G_{1})$ with every vertex of $I(G_{2})$.
\end{defn}
\begin{defn}\cite{15}\label{lem4}
The subdivision$-$vertex join of two vertex disjoint graphs $G_{1}$ and $G_{2}$, denoted by $G_{1}\dot{\vee}G_{2}$, is the graph obtained from $S(G_{1})$ and $S(G_{2})$ by joining each vertex of $V(G_{1})$ with every vertex of $V(G_{2})$.
\end{defn}
\begin{defn}\cite{15}\label{lem5}
The subdivision$-$(vertex$-$edge) join of two vertex disjoint graphs $G_{1}$ and $G_{2}$, denoted by $G_{1}\underline{\vee}G_{2}$, is the graph obtained from $S(G_{1})$ and $S(G_{2})$ by joining each vertex of $V(G_{1})$ with every vertex of $I(G_{2})$.
\end{defn}

\begin{defn}\cite{12}\label{lem6}
Let $G$ and $H$ be two graphs on vertex sets $V(G)=\{u_{1},u_{2},\cdots,u_{p}\}$
and $V(H)=\{v_{1},v_{2},\cdots,v_{n}\}$, respectively. Then their lexicographic product $G[H]$ is
a graph with vertex set $V(G[H])=V(G)\times V(H)$, in which $u=(u_{1},v_{1})$ is adjacent to $v=(u_{2},v_{2})$ if and only if either\\
$(a)$ $u_{1}$ is adjacent to $u_{2}$ in $G$, or\\
$(b)$ $u_{1}=u_{2}$ and $v_{1}$ is adjacent to $v_{2}$ in $H$.
\end{defn}

\begin{defn}\cite{12}\label{lem7}
Let $G$ and $H$ be two graphs on vertex sets $V(G)=\{u_{1},u_{2},\cdots,u_{p}\}$
and $V(H)=\{v_{1},v_{2},\cdots,v_{n}\}$, respectively. Then their cartesian product $G+H$ is
a graph with vertex set $V(G+H)=V(G)\times V(H)$, in which $u=(u_{1},v_{1})$ is adjacent to $v=(u_{2},v_{2})$ if and only if either\\
$(a)$ $u_{1}=u_{2}$ and $v_{1}$ is adjacent to $v_{2}$ in $H$, or\\
$(b)$ $v_{1}=v_{2}$ and $u_{1}$ is adjacent to $u_{2}$ in $G$.
\end{defn}

\begin{defn}\cite{16}\label{lem8}
The Hamming graph Ham$(d,n)$, $d\geq2, n\geq2$, of diameter $d$ and characteristic $n$ have vertex set consisting of all $d$-tuples of elements taken from an $n$-element set, with two vertices adjacent if and only if they differ in exactly one coordinate. Ham$(d,n)$ is equal to $K_{n}+K_{n}+\cdots+K_{n}$, the cartesian product of $K_{n}$, the complete graph on $n$ vertices, $d$ times. Ham$(3,n)$ is referred to as a cubic lattice graph.
\end{defn}
\begin{lem}\cite{17}\label{lem9}
Let
\begin{equation} 
A=
\left[                 
  \begin{array}{cc}     
  \begin{smallmatrix}
    A_{0} & A_{1} \\  
    A_{1} & A_{0}\\  
  \end{smallmatrix}
  \end{array}
\right],    \nonumber             
\end{equation}
be a $2\times2$ block symmetric matrix. Then the eigenvalues of $A$ are those of $A_{0} + A_{1}$ together with those of $A_{0}-A_{1}$ .
\end{lem}
\begin{lem}\cite{12}\label{lem10}
 Let $G$ be an $r-$regular graph with adjacency matrix $A$ and incidence matrix $R$, and
the line graph $L(G)$ of $G$. Then $RR^{T} = A+rI$, $R^{T}R = A(L(G))+2I$. Also, if $J$ is an all-one matrix of appropriate order, then
$JR = 2J = R^{T}J$ and $JR^{T} = rJ = RJ$.
\end{lem}
\begin{lem}\cite{12}\label{lem11}
Let $G$ be an $r-$regular graph on $p$ vertices and $q$ edges $(q=\frac{1}{2}pr)$ with the adjacency spectrum $\{r,\lambda_{2},\cdots,\lambda_{p}\}$. Then, the adjacency spectrum of line graph of graph $G$ is

\begin{equation} 
spec(L(G))=
\left[                 
  \begin{array}{ccccc}     
    2r-2 & \lambda_{2}+r-2 & \cdots & \lambda_{p}+r-2 & -2\\  
    1 & 1 & \cdots & 1 & q-p\\  
  \end{array}
\right],    \nonumber             
\end{equation}
Also, $Z$ is an eigenvector corresponding to the eigenvalue $-2$ if and only if $RZ = 0$ where $R$ is the incidence matrix of $G$.
\end{lem}

\begin{lem}\cite{18}\label{lem12}
If $A$ is an $n\times n$ nonnegative matrix with the spectral radius $\lambda(A)$
and row sums $r_{1},r_{2},\cdots,r_{n}$, then
\begin{equation*}
\begin{aligned}
\min_{1\leq i\leq n}r_{i}\leq\lambda(A)\leq\max_{1\leq i\leq n}r_{i}.
\end{aligned}
\end{equation*}
Moreover, if $A$ is
irreducible, then one of the equalities holds if and only if the row sums of $A$ are all equal.
\end{lem}

Let $F_{1}$ be the $5$-vertex path, $F_{2}$ the graph obtained by identifying a vertex of a triangle
with an end vertex of the $3$-vertex path, and $F_{3}$ the graph obtained by identifying a
vertex of a triangle with a vertex of another triangle (see Fig. $1$).
\begin{lem}\cite{19,20}\label{lem13}
For a connected graph $G$, $diam(L(G))\leq2$ if and only if none
of the three graphs $F_{1}$, $F_{2}$ and $F_{3}$ of Fig. $1$ is an induced subgraph of $G$.
\end{lem}

\begin{figure}[htbp]
\centering
\subfigure[$F_1$]{\includegraphics[width=1.45in]{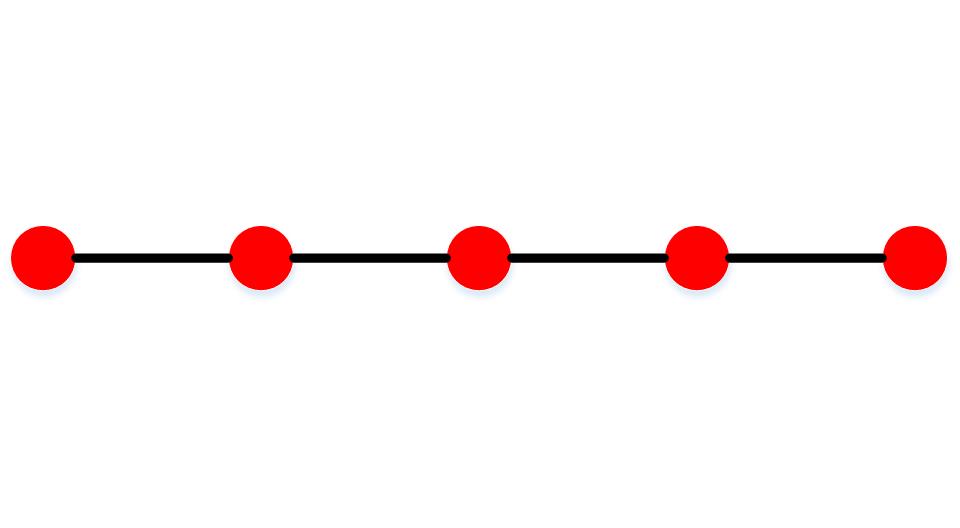}}
\hspace{0.4in}
\subfigure[$F_2$]{ \includegraphics[width=1.35in]{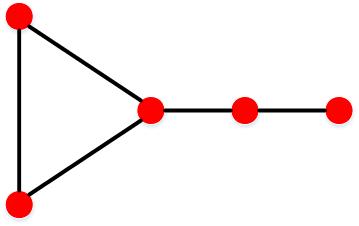}}
\hspace{0.35in}
\subfigure[$F_3$]{ \includegraphics[width=1.15in]{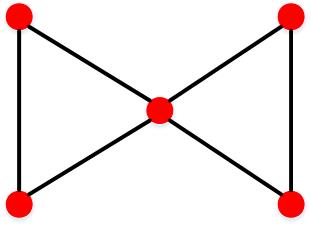}}
\hspace{0.35in}
\caption{}
\label{Fig1}
\end{figure}

\begin{lem}\cite{21}\label{lem14}
Let $G$ be a $k$-transmission regular graph of order $p$ with distance spectrum
$\{k,\mu_{2}^{\mathcal{D}},\mu_{3}^{\mathcal{D}},\cdots,\mu_{p}^{\mathcal{D}}\}$. Also let $H$ be a $t$-transmission regular graph of order $n$ with distance
spectrum $\{t,\eta_{2}^{\mathcal{D}},\eta_{3}^{\mathcal{D}},\cdots,\eta_{n}^{\mathcal{D}}\}$. Then the generalized distance spectrum of $G+H$ consists precisely of\\
$(1)$ $nk+pt$;\\
$(2)$ $(1-\alpha)n\mu_{i}^{\mathcal{D}}+\alpha(nk+pt)$, for $i=2,3,\cdots,p$; \\
$(3)$ $(1-\alpha)p\eta_{j}^{\mathcal{D}}+\alpha(nk+pt)$, for $j=2,3,\cdots,n$;\\
$(4)$ $\alpha(nk+pt)$, with multiplicity $(p-1)(n-1)$.
\end{lem}

\begin{lem}\cite{22}\label{lem15}
The distance spectrum of the cycle $C_{n}$ is given by
\renewcommand\arraystretch{2}
\begin{table}[htbp]
\centering
\begin{tabular}{p{1cm}<{\centering}|p{3cm}<{\centering}|p{2.5cm}<{\centering}|p{2.5cm}<{\centering}}
\hline
$n$ & greast eigenvalue & $j$ even & $j$ odd \\ \hline
even & $\frac{n^{2}}{4}$ & $0$ & $-\csc^{2}(\frac{\pi j}{n})$\\ \hline
odd & $\frac{n^{2}-1}{4}$ & $-\frac{1}{4}\sec^{2}(\frac{\pi j}{2n})$ &  $-\frac{1}{4}\csc^{2}(\frac{\pi j}{2n})$\\ \hline
\end{tabular}
\end{table}
\end{lem}
\section{The $\mathcal{D}-$spectrum, $\mathcal{L}-$spectrum and $\mathcal{Q}-$spectrum of some composite graphs}
\subsection{The $\mathcal{L}-$spectrum and $\mathcal{Q}-$spectrum of $G\{K_{m}\}$}
Let $G$ be a distance regular graph with vertex set $V=\{v_{1},v_{2},\cdots,v_{n}\}$ and let the vertex set of the $i$th copy of $K_{m}$ be $\{u_{1}^{i},u_{2}^{i},\cdots,u_{m}^{i}\}$ with root $u_{1}^{i}$. Let $W_{j}=\{u_{j}^{1},u_{j}^{2},\cdots,u_{j}^{n}\}$. With this labeling, $V(G\{K_{m}\})=V(G)\bigcup W_{1}\bigcup W_{2}\cdots \bigcup W_{m}$.

 According to Definiton \ref{lem1}, its distance Laplacian matrix $\mathcal{L}(G\{K_{m}\})$ can be written in the form
 \begin{equation} 
\mathcal{L}(G\{K_{m}\})=
\left[                 
  \begin{array}{ccc}     
  \begin{smallmatrix}
    (2mn-n)I_{n}+mTr_{G}+\mathcal{L}(G) & -(\mathcal{D}(G)+J_{n}) & -J_{1\times m-1}\otimes(\mathcal{D}(G)+2J_{n}) \\  
    -(\mathcal{D}(G)+J_{n}) & (3mn-2m+2)I_{n}+mTr_{G}+\mathcal{L}(G)-2J_{n} &  -J_{1\times m-1}\otimes(\mathcal{D}(G)+3J_{n}-2I_{n})\\  
    -J_{m-1\times 1}\otimes(\mathcal{D}(G)+2J_{n}) & -J_{m-1\times 1}\otimes(\mathcal{D}(G)+3J_{n}-2I_{n}) & \mathcal{L}^{\ast}
  \end{smallmatrix}
  \end{array}
\right],    \nonumber             
\end{equation}
where
$\mathcal{L}^{\ast}=(I-J)_{m-1}\otimes(Tr_{G}-\mathcal{L}(G)+4J_{n}-3I_{n})+I_{m-1}\otimes[(4mn-3m+n+4)I_{n}+mTr_{G}+\mathcal{L}(G)-4J_{n}]$, $J$ is the all-one matrix, and $I$ is the identity matrix of appropriate orders.

According to Definiton \ref{lem1}, its distance signless Laplacian matrix $\mathcal{Q}(G\{K_{m}\})$ can be written in the form
 \begin{equation} 
\mathcal{Q}(G\{K_{m}\})=
\left[                 
  \begin{array}{ccc}     
  \begin{smallmatrix}
    (2mn-n)I_{n}+mTr_{G}+\mathcal{Q}(G) & \mathcal{D}(G)+J_{n} & J_{1\times m-1}\otimes(\mathcal{D}(G)+2J_{n}) \\  
    \mathcal{D}(G)+J_{n} & (3mn-2m-2)I_{n}+mTr_{G}+\mathcal{Q}(G)+2J_{n} &  J_{1\times m-1}\otimes(\mathcal{D}(G)+3J_{n}-2I_{n})\\  
    J_{m-1\times 1}\otimes(\mathcal{D}(G)+2J_{n}) & J_{m-1\times 1}\otimes(\mathcal{D}(G)+3J_{n}-2I_{n}) & \mathcal{Q}^{\ast}
  \end{smallmatrix}
  \end{array}
\right],   \nonumber             
\end{equation}
where
$\mathcal{Q}^{\ast}=(J-I)_{m-1}\otimes(\mathcal{Q}(G)-Tr_{G}+4J_{n}-3I_{n})+I_{m-1}\otimes[(4mn-3m+n-4)I_{n}+mTr_{G}+\mathcal{Q}(G)+4J_{n}]$, $J$ is the all-one matrix, and $I$ is the identity matrix of appropriate orders. Now we shall find the $\mathcal{L}-$spectrum and $\mathcal{Q}-$spectrum of $\mathcal{Q}(G\{K_{m}\})$.

\begin{thm}\label{thm1}
Let $G$ be a distance regular graph with distance regularity $k$, a distance Laplacian matrix $\mathcal{L}(G)$ and distance Laplacian spectrum $\{0=\mu_{1}^{\mathcal{L}},\mu_{2}^{\mathcal{L}},\cdots,\mu_{n}^{\mathcal{L}}\}$. If $K_{m}$ is a complete graph of $m$ vertices, then the distance Laplacian spectrum of $\mathcal{L}(G\{K_{m}\})$ is

(a) $(m+1)k+4mn-3m+n+1$ with multiplicity $(m-2)n$,

(b) the roots of the equation $\prod_{i=2}^{n}(x^{3}-Ax^{2}+BX+C)=0$, for $\mu_{i}^{\mathcal{L}}\neq0$,\\
where $A=2k - 2m + \mu_{i}^{\mathcal{L}} + 2km + 9mn + m\mu_{i}^{\mathcal{L}}$, $B=2n - 2km - 2kn + 2mn + 2k\mu_{i}^{\mathcal{L}} - 2m\mu_{i}^{\mathcal{L}} + 2n\mu_{i}^{\mathcal{L}} - 2km^2 + 2k^2m - 2mn^2 - 12m^2n - 2m^2\mu_{i}^{\mathcal{L}} + k^2 - n^2 + k^2m^2 + 26m^2n^2 + 13km^2n + 2km^2\mu_{i}^{\mathcal{L}} + 5m^2n\mu_{i}^{\mathcal{L}} + 11kmn + 4km\mu_{i}^{\mathcal{L}} + 7mn\mu_{i}^{\mathcal{L}}$, and $C= -12km^2n\mu_{i}^{\mathcal{L}} - 5km^3n\mu_{i}^{\mathcal{L}} - 9kmn\mu_{i}^{\mathcal{L}} - 2kn\mu_{i}^{\mathcal{L}} + 2km\mu_{i}^{\mathcal{L}} - 4kmn - 11m^2n^2\mu_{i}^{\mathcal{L}} - 6m^3n^2\mu_{i}^{\mathcal{L}} - 3k^2m^2\mu_{i}^{\mathcal{L}} - k^2m^3\mu_{i}^{\mathcal{L}} - 20km^3n^2 - 13km^2n^2 - 6k^2m^2n - 4k^2m^3n + 6m^2n\mu_{i}^{\mathcal{L}} + 4m^3n\mu_{i}^{\mathcal{L}}- 4mn^2\mu_{i}^{\mathcal{L}} + 4km^2\mu_{i}^{\mathcal{L}} - 3k^2m\mu_{i}^{\mathcal{L}} + 2km^3\mu_{i}^{\mathcal{L}} + 8km^3n + 7kmn^2 + 4km^2n - 24m^3n^3 + 16m^3n^2 - 8m^2n^2 + 6m^2n^3 - k^2\mu_{i}^{\mathcal{L}} + 3mn^3 - 4mn^2 + 2k^2n - 2n\mu_{i}^{\mathcal{L}} + n^2\mu_{i}^{\mathcal{L}} + 2n^2$.

(c) $\frac{2(k-m+n+km)+5mn\pm\sqrt{m^{2}n^{2}-4m^{2}n+4m^{2}+8n^{2}-8n}}{2}$, $0$.
\end{thm}

\begin{proof}By the definition of the cluster of two graphs, the distance Laplacian matrix $G\{K_{m}\}$ can be written in the form

\begin{equation} 
\mathcal{L}(G\{K_{m}\})=
\left[                 
  \begin{array}{ccc}     
  \begin{smallmatrix}
    (2mn-n)I_{n}+mTr_{G}+\mathcal{L}(G) & -(\mathcal{D}(G)+J_{n}) & -J_{1\times m-1}\otimes(\mathcal{D}(G)+2J_{n}) \\  
    -(\mathcal{D}(G)+J_{n}) & (3mn-2m+2)I_{n}+mTr_{G}+\mathcal{L}(G)-2J_{n} &  -J_{1\times m-1}\otimes(\mathcal{D}(G)+3J_{n}-2I_{n})\\  
    -J_{m-1\times 1}\otimes(\mathcal{D}(G)+2J_{n}) & -J_{m-1\times 1}\otimes(\mathcal{D}(G)+3J_{n}-2I_{n}) & \mathcal{L}^{\ast}
  \end{smallmatrix}
  \end{array}
\right],    \nonumber             
\end{equation}
where
$\mathcal{L}^{\ast}=(I-J)_{m-1}\otimes(Tr_{G}-\mathcal{L}(G)+4J_{n}-3I_{n})+I_{m-1}\otimes[(4mn-3m+n+4)I_{n}+mTr_{G}+\mathcal{L}(G)-4J_{n}]$, $J$ is the all-one matrix, and $I$ is the identity matrix of appropriate orders.

Let $Y_{j},j=2,3,\cdots,m-1$ be the eigenvectors of $J_{m-1}$ corresponding to zero, then $Y_{j}$ is orthogonal to the all-ones vector. Let $e_{n\times1}^{l}$ is a $n\times1$ column vector with the $l$th entry equal $1$, and all other entrie equal to zero.

To prove part(a), consider the vector
$\begin{matrix}
\Psi_{j}^{l}=
\left[                 
  \begin{array}{ccc}   
    \begin{smallmatrix}
    0_{n\times1}  \\  
    0_{n\times1} \\  
    Y_{j}\otimes e_{n\times1}^{l}
    \end{smallmatrix}
  \end{array}
\right]    \nonumber             
\end{matrix}$, which
is an eigenvector of $\mathcal{L}(G\{K_{m}\})$ with eigenvalues $(m+1)k+4mn-3m+n+1$ for each $j=2,3,\cdots,m-1$ and $l=1,2,\cdots,n$.
\begin{equation*}
\begin{aligned} 
\mathcal{L}(G\{K_{m}\})\cdot\Psi_{j}^{l}&=
\left[                 
  \begin{array}{ccc}     
  \begin{smallmatrix}
    (2mn-n)I_{n}+mTr_{G}+\mathcal{L}(G) & -(\mathcal{D}(G)+J_{n}) & -J_{1\times m-1}\otimes(\mathcal{D}(G)+2J_{n}) \\  
    -(\mathcal{D}(G)+J_{n}) & (3mn-2m+2)I_{n}+mTr_{G}+\mathcal{L}(G)-2J_{n} &  -J_{1\times m-1}\otimes(\mathcal{D}(G)+3J_{n}-2I_{n})\\  
    -J_{m-1\times 1}\otimes(\mathcal{D}(G)+2J_{n}) & -J_{m-1\times 1}\otimes(\mathcal{D}(G)+3J_{n}-2I_{n}) & \mathcal{L}^{\ast}
  \end{smallmatrix}
  \end{array}
\right]    \nonumber      \\&       
\begin{matrix}
\times
\left[                 
  \begin{array}{ccc}   
    \begin{smallmatrix}
    0_{n\times1}  \\  
    0_{n\times1} \\  
    Y_{j}\otimes e_{n\times1}^{l}
    \end{smallmatrix}
  \end{array}
\right]    \nonumber             
\end{matrix},\\  &
\begin{matrix}
=
\left[                 
  \begin{array}{ccc}   
    \begin{smallmatrix}
    (-J_{1\times m-1}\otimes(\mathcal{D}(G)+2J_{n}))\cdot Y_{j}\otimes e_{n\times1}^{l} \\  
    (-J_{1\times m-1}\otimes(\mathcal{D}(G)+3J_{n}-2I_{n}))\cdot Y_{j}\otimes e_{n\times1}^{l}\\  
    \mathcal{L}^{\ast}\cdot Y_{j}\otimes e_{n\times1}^{l}
    \end{smallmatrix}
  \end{array}
\right]    \nonumber             
\end{matrix},\\&
\begin{matrix}
=
\left[                 
  \begin{array}{ccc}   
    \begin{smallmatrix}
    0 \\  
    0\\  
    [(m+1)k+4mn-3m+n+1]\cdot Y_{j}\otimes e_{n\times1}^{l}
    \end{smallmatrix}
  \end{array}
\right],    \nonumber             
\end{matrix}\\  &
=[(m+1)k+4mn-3m+n+1]\Psi_{j}^{l}.
\end{aligned}
\end{equation*}
Thus, $(m+1)k+4mn-3m+n+1$ is an eigenvalue of $\mathcal{L}(G\{K_{m}\})$ with multiplicity $n(m-2)$.

To prove part(b), now consider the eigenvalue $\mu_{i}^{\mathcal{L}}\neq0$ of $\mathcal{L}(G)$ with an eigenvector $X_{i},i=2,3,\cdots,n$. Then, $X_{i}$ is orthogonal to the all-ones vector. Let $\mu_{i_{r}},r=1,2,3$ be the three roots of the equation
\begin{equation}
x^{3}-Ax^{2}+BX+C=0\tag{1},
\end{equation}
where $A=2k - 2m + \mu_{i}^{\mathcal{L}} + 2km + 9mn + m\mu_{i}^{\mathcal{L}}$, $B=2n - 2km - 2kn + 2mn + 2k\mu_{i}^{\mathcal{L}} - 2m\mu_{i}^{\mathcal{L}} + 2n\mu_{i}^{\mathcal{L}} - 2km^2 + 2k^2m - 2mn^2 - 12m^2n - 2m^2\mu_{i}^{\mathcal{L}} + k^2 - n^2 + k^2m^2 + 26m^2n^2 + 13km^2n + 2km^2\mu_{i}^{\mathcal{L}} + 5m^2n\mu_{i}^{\mathcal{L}} + 11kmn + 4km\mu_{i}^{\mathcal{L}} + 7mn\mu_{i}^{\mathcal{L}}$, and $C= -12km^2n\mu_{i}^{\mathcal{L}} - 5km^3n\mu_{i}^{\mathcal{L}} - 9kmn\mu_{i}^{\mathcal{L}} - 2kn\mu_{i}^{\mathcal{L}} + 2km\mu_{i}^{\mathcal{L}} - 4kmn - 11m^2n^2\mu_{i}^{\mathcal{L}} - 6m^3n^2\mu_{i}^{\mathcal{L}} - 3k^2m^2\mu_{i}^{\mathcal{L}} - k^2m^3\mu_{i}^{\mathcal{L}} - 20km^3n^2 - 13km^2n^2 - 6k^2m^2n - 4k^2m^3n + 6m^2n\mu_{i}^{\mathcal{L}} + 4m^3n\mu_{i}^{\mathcal{L}}- 4mn^2\mu_{i}^{\mathcal{L}} + 4km^2\mu_{i}^{\mathcal{L}} - 3k^2m\mu_{i}^{\mathcal{L}} + 2km^3\mu_{i}^{\mathcal{L}} + 8km^3n + 7kmn^2 + 4km^2n - 24m^3n^3 + 16m^3n^2 - 8m^2n^2 + 6m^2n^3 - k^2\mu_{i}^{\mathcal{L}} + 3mn^3 - 4mn^2 + 2k^2n - 2n\mu_{i}^{\mathcal{L}} + n^2\mu_{i}^{\mathcal{L}} + 2n^2$.


For each $i=2,3,\cdots,n$, the roots of $Eq.(1)$,$\mu_{i_{r}},r=1,2,3$ are eigenvalues of $\mathcal{L}(G\{K_{m}\})$.

To prove the claim, we investigate the condition under which
$\begin{matrix}
\Phi_{i}^{r}=
\left[                 
  \begin{array}{ccc}   
    \begin{smallmatrix}
    t_{r}X_{i} \\  
    X_{i}\\  
    J_{m-1\times1\otimes s_{r}X_{i}}
    \end{smallmatrix}
  \end{array}
\right]    \nonumber             
\end{matrix}$
becomes an eigenvector corresponding to $\mu_{i_{r}},r=1,2,3$ for $\mathcal{L}(G\{K_{m}\})$.

Now using $\mathcal{L}(G\{K_{m}\})\cdot \Phi_{i}^{r}=\mu_{i_{r}}\cdot\Phi_{i}^{r}$ and $X_{i}\neq0$, we get the following.
\begin{equation}
(2mn-n+mk+\mu_{i}^{\mathcal{L}})t_{r}-(m-1)(k-\mu_{i}^{\mathcal{L}})s_{r}-(k-\mu_{i}^{\mathcal{L}})=t_{r}\mu_{i_{r}} \tag{2},
\end{equation}
\begin{equation}
(\mu_{i}^{\mathcal{L}}-k)t_{r}-(m-1)(k-\mu_{i}^{\mathcal{L}}-2)s_{r}+3mn-2m+mk+\mu_{i}^{\mathcal{L}}+2=\mu_{i_{r}} \tag{3},
\end{equation}
\begin{equation}
(\mu_{i}^{\mathcal{L}}-k)t_{r}+[(m-1)\mu_{i}^{\mathcal{L}}+4mn+n+2k-2]s_{r}+\mu_{i}^{\mathcal{L}}-k+2=s_{r}\mu_{i_{r}} \tag{4}.
\end{equation}
Now solving $Eqs.(2)-(4)$ by Matlab yields a cubic equation about $\mu_{i_{r}}$, which is equivalent to $Eq.(1)$, proving our claim.
Thus, forming eigenvectors of this type we get $n(m-2)+3(n-1)=mn+n-3$ eigenvectors, and there remains $3$.


To prove part(c), now consider the eigenvalue $\mu_{i}^{\mathcal{L}}=0$ of $\mathcal{L}(G)$ with an eigenvector $J_{n\times1}$.
 By the construction, all eigenvectors are orthogonal to the all-ones vectors, and hence the remaining three are of the form
$\begin{matrix}
\Omega=
\left[                 
  \begin{array}{ccc}   
    \begin{smallmatrix}
    \alpha J_{n\times1} \\  
    \beta J_{n\times1}\\  
    J_{m-1\times1}\otimes \gamma J_{n\times1}
    \end{smallmatrix}
  \end{array}
\right]    \nonumber             
\end{matrix}$
for some $(\alpha,\beta,\gamma)\neq(0,0,0)$.

Let $v$ be an eigenvalue with an eigenvector $\Omega$, then the equation $\mathcal{L}(G\{K_{m}\})\cdot \Omega=v\cdot\Omega$ gives the following.
\begin{equation}
(2mn-n+mk)\alpha-(k+n)\beta-(m-1)(k+2n)\gamma=v\alpha \tag{5},
\end{equation}
\begin{equation}
-(k+n)\alpha+(3mn-2m+2+mk-2n)\beta-(m-1)(k+3n-2)\gamma=v\beta \tag{6},
\end{equation}
\begin{equation}
-(k+2n)\alpha-(k+3n-2)\beta+(2k+5n-2)\gamma=v\gamma \tag{7}.
\end{equation}

Now $\alpha\neq0$. Otherwise, solving $Eqs.(5)-(7)$ implies $\beta=0$ and $\gamma=0$. Therefore, without loss of generality we can assume that $\alpha=1$ and solving $Eqs.(5)-(7)$ by Matlab yields the part$(c)$. Whence the theorem.
\qed\end{proof}
\begin{cor}\label{thm2}
 Let $G$ be a distance regular graph with distance regularity $k$, a distance Laplacian matrix $\mathcal{L}(G)$ and distance Laplacian spectrum $\{0=\mu_{1}^{\mathcal{L}},\mu_{2}^{\mathcal{L}},\cdots,\mu_{n}^{\mathcal{L}}\}$. If $K_{m}$ is a isolated vertex, then the distance Laplacian spectrum of $\mathcal{L}(G\{K_{1}\})$ is

(a) $k+2n+\mu_{i}^{\mathcal{L}}\pm\sqrt{(k-\mu_{i}^{\mathcal{L}})^{2}+n^{2}}$ for $\mu_{i}^{\mathcal{L}}\neq0$,

(b) $2(k+n)$, $0$.
\end{cor}
\begin{proof}The proof is immediate from the above theorem by choosing $m=1$.\qed\end{proof}
\begin{thm}\label{thm3}
Let $G$ be a distance regular graph with distance regularity $k$, a distance signless Laplacian matrix $\mathcal{Q}(G)$ and distance signless Laplacian spectrum $\{2k=\mu_{1}^{\mathcal{Q}},\mu_{2}^{\mathcal{Q}},\cdots,\mu_{n}^{\mathcal{Q}}\}$. If $K_{m}$ is a complete graph of $m$ vertices, then the distance signless Laplacian spectrum of $\mathcal{Q}(G\{K_{m}\})$ is

(a) $(m+1)k+4mn-3m+n-1$ with multiplicity $(m-2)n$,

(b) the roots of the equation $\prod_{i=2}^{n}(x^{3}-Ax^{2}+BX+C)=0$ for $\mu_{i}^{\mathcal{Q}}\neq2k$,\\
where $A=2k - 8m + \mu_{i} + 2km + 9mn + m\mu_{i}^{\mathcal{Q}}$, $B=4m - 2n - 6km - 2kn + 4mn + 2k\mu_{i}^{\mathcal{Q}} - 10m\mu_{i}^{\mathcal{Q}} + 2n\mu_{i}^{\mathcal{Q}} - 14km^2 + 2k^2m - 2mn^2 - 42m^2n - 2m^2\mu_{i}^{\mathcal{Q}} + k^2 + 12m^2 - n^2 + k^2m^2 + 26m^2n^2 + 13km^2n + 2km^2\mu_{i}^{\mathcal{Q}} + 5m^2n\mu_{i}^{\mathcal{Q}} + 11kmn + 4km\mu_{i}^{\mathcal{Q}} + 7mn\mu_{i}^{\mathcal{Q}}
$, and $C=- 12km^2n\mu_{i}^{\mathcal{Q}} - 5km^3n\mu_{i}^{\mathcal{Q}} - 9kmn\mu_{i}^{\mathcal{Q}} + 2mn\mu_{i}^{\mathcal{Q}} - 2kn\mu_{i}^{\mathcal{Q}} + 10km\mu_{i}^{\mathcal{Q}} - 4kmn - 11m^2n^2\mu_{i}^{\mathcal{Q}} - 6m^3n^2\mu_{i}^{\mathcal{Q}} - 3k^2m^2\mu_{i}^{\mathcal{Q}} - k^2m^3\mu_{i}^{\mathcal{Q}} - 20km^3n^2 - 13km^2n^2 - 6k^2m^2n - 4k^2m^3n + 28m^2n\mu_{i}^{\mathcal{Q}} + 4m^3n\mu_{i}^{\mathcal{Q}} - 4mn^2\mu_{i}^{\mathcal{Q}} + 12km^2\mu_{i}^{\mathcal{Q}} - 3k^2m\mu_{i}^{\mathcal{Q}} + 2km^3\mu_{i}^{\mathcal{Q}} + 38km^3n + 10km^2n + 7kmn^2 + 52m^3n^2 - 24m^3n^3 - 18m^2n^2 + 6m^2n^3 + 6k^2m^3 + 4k^2m^2 - 12m^2\mu_{i}^{\mathcal{Q}} - k^2\mu_{i}^{\mathcal{Q}} - 24m^3n + 4m^2n + 3mn^3 + 2k^2n - 12km^3 - 4km^2 - 2k^2m + 2nu - 4m\mu_{i}^{\mathcal{Q}} + 4mn + n^2\mu_{i}^{\mathcal{Q}} - 2n^2$.

(c) the roots of the equation $x^{3}-A^{\ast}x^{2}+B^{\ast}X+C^{\ast}=0$,\\
where $A^{\ast}=4k - 8m - 2n + 4km + 13mn$, $B^{\ast}=4m - 2n - 26km - 2kn + 4mn - 18km^2 + 10k^2m - 19mn^2 - 50m^2n + 5k^2 + 12m^2 + 7n^2 + 5k^2m^2 + 46m^2n^2 + 31km^2n + 29kmn$, and $C^{\ast}=- 52km^3n^2 - 38km^2n^2 - 36k^2m^2n - 18k^2m^3n + 74km^2n + 54km^3n - 18k^2mn + 8kmn^2 + 68m^3n^2 - 48m^3n^3 + 40m^2n^3 - 34m^2n^2 + 28k^2m^2 + 10k^2m^3 - 6k^3m^2 - 2k^3m^3 - 24m^3n - 16mn^3 + 4m^2n + 2mn^2 - 6kn^2 - 28km^2 + 18k^2m - 12km^3 - 6k^3m + 4mn + 4kn - 8km + 4n^3 - 2k^3$.
\end{thm}
\begin{proof}we can prove the theorem by using the similar method as for Theorem \ref{thm1}.\qed\end{proof}

\begin{cor}\label{thm4}
Let $G$ be a distance regular graph with distance regularity $k$, a distance signless Laplacian matrix $\mathcal{Q}(G)$ and distance signless Laplacian spectrum $\{2k=\mu_{1}^{\mathcal{Q}},\mu_{2}^{\mathcal{Q}},\cdots,\mu_{n}^{\mathcal{Q}}\}$. If $K_{m}$ is a isolated vertex, then the distance signless Laplacian spectrum of $\mathcal{Q}(G\{K_{1}\})$ is

(a) $k+2n-2+\mu_{i}\pm\sqrt{(k-\mu_{i}^{\mathcal{Q}})^{2}+(n-2)^{2}}$ for $\mu_{i}^{\mathcal{Q}}\neq2k$,

(b) $3k+3n-2\pm\sqrt{(k+n)^{2}+(2n-2)^{2}}$.
\end{cor}
\begin{proof}The proof is immediate from the above theorem by choosing $m=1$.\qed\end{proof}

\subsection{The $\mathcal{L}-$spectrum and $\mathcal{Q}-$spectrum of $D_{2}G$}
\begin{thm}\label{thm5}
Let $G$ be a graph with distance Laplacian spectrum $\{\mu_{1}^{\mathcal{L}}\geq\mu_{2}^{\mathcal{L}}\geq\cdots\geq\mu_{n}^{\mathcal{L}}=0\}$, then the distance Laplacian spectrum of $D_{2}G$ consists of eigenvalues
$0, 2\mu_{i}^{\mathcal{L}}$ for $i=1,2,\cdots,n-1$ and $2Tr_{G}(v_{i})+4$ for $i=1,2,\cdots,n$, where $Tr_{G}(v_{i})$ is the sum of all coordinates of the row vector of $\mathcal{D}(G)$ indexed by $v_{i}$.
\end{thm}
\begin{proof}According to Definition \ref{lem2}, we have:
\begin{equation*}
d_{D_{2}G}(v_{i},v_{j})=d_{G}(v_{i},v_{j})
\end{equation*}
\begin{equation*}
d_{D_{2}G}(v_{i},u_{i})=2
\end{equation*}
\begin{equation*}
d_{D_{2}G}(v_{i},u_{j})=d_{G}(v_{i},v_{j})
\end{equation*}
\begin{equation*}
d_{D_{2}G}(v_{j},u_{i})=d_{G}(v_{j},v_{i})
\end{equation*}
Hence the distance Laplacian matrix $D_{2}G$ can be written in the form

\begin{equation} 
\mathcal{L}(D_{2}G)=
\left[                 
  \begin{array}{cc}     
  \begin{smallmatrix}
    2Tr_{G}+2I_{n}-\mathcal{D}(G) & -(\mathcal{D}(G)+2I_{n}) \\  
    -(\mathcal{D}(G)+2I_{n}) & 2Tr_{G}+2I_{n}-\mathcal{D}(G)\\  
  \end{smallmatrix}
  \end{array}
\right],    \nonumber             
\end{equation}
and the theorem follows from Lemma \ref{lem9}.\qed\end{proof}
\begin{cor}\label{thm6}
Let $G$ be a $k$-distance regular graph with distance Laplacian spectrum $\{\mu_{1}^{\mathcal{L}}\geq\mu_{2}^{\mathcal{L}}\geq\cdots\geq\mu_{n}^{\mathcal{L}}=0\}$, then the distance Laplacian spectrum of $D_{2}G$ consists of eigenvalues
$0, 2\mu_{i}^{\mathcal{L}}$ for $i=1,2,\cdots,n-1$ and $2k+4$ with multiplicity $n$.
\end{cor}
\begin{proof}The proof follows from the above theorem by choosing $Tr_{G}(v_{i})=k$.\qed\end{proof}
\begin{cor}\label{thm7}
Let $G$ be a $k$-distance regular graph with distance Laplacian spectrum $\{\mu_{1}^{\mathcal{L}}\geq\mu_{2}^{\mathcal{L}}\geq\cdots\geq\mu_{n}^{\mathcal{L}}=0\}$, then the distance Laplacian energy of $D_{2}G$ is
\begin{equation*}
DLE(D_{2}G)=2\sum_{i=1}^{\sigma}\mu_{i}^{\mathcal{L}}-2\sum_{i=\sigma+1}^{n-1}\mu_{i}^{\mathcal{L}}+2nk-4\sigma k+4n-4\sigma.
\end{equation*}
\end{cor}
\begin{proof}By Corollary \ref{thm6}, we know that the distance Laplacian spectrum of $D_{2}G$. We have $\sum_{i=1}^{n}\mu_{i}^{\mathcal{L}}(G)=\sum_{i=1}^{n}Tr_{G}(v_{i})=nt(G)$, then we get $t(D_{2}G)=2k+2$. Let $\sigma$ be the largest positive integer such that $\mu_{\sigma}^{\mathcal{L}}(G)\geq k+1$, which is equivalent to $2\mu_{i}^{\mathcal{L}}\geq t(G)$ for $i=1,2,\cdots,\sigma$; otherwise, $2\mu_{i}^{\mathcal{L}}<t(G)$ for $i=\sigma+1,\sigma+2,\cdots,n-1$. Starting from the definition of distance Laplacian energy, the proof of Corollary is completed.\qed\end{proof}

\begin{thm}\label{thm8}
Let $G$ be a graph with distance signless Laplacian spectrum $\{\mu_{1}^{\mathcal{Q}},\mu_{2}^{\mathcal{Q}},\cdots,\mu_{n}^{\mathcal{Q}}\}$, then the distance signless Laplacian spectrum of $D_{2}G$ consists of eigenvalues $2\mu_{i}^{\mathcal{Q}}+4$ and $2Tr_{G}(v_{i})$ for $i=1,2,\cdots,n$,  where $Tr_{G}(v_{i})$ is the sum of all coordinates of the row vector of $\mathcal{D}(G)$ indexed by $v_{i}$.
\end{thm}
\begin{proof}By Definition \ref{lem2}, the distance signless Laplacian matrix of $D_{2}G$ can be written in the form
\begin{equation} 
Q(D_{2}G)=
\left[                 
  \begin{array}{cc}     
  \begin{smallmatrix}
    2Tr_{G}+2I_{n}+\mathcal{D}(G) & \mathcal{D}(G)+2I_{n} \\  
   \mathcal{D}(G)+2I_{n} & 2Tr_{G}+2I_{n}+\mathcal{D}(G)\\  
  \end{smallmatrix}
  \end{array}
\right],    \nonumber             
\end{equation}
and the theorem follows from Lemma \ref{lem9}.\qed\end{proof}
\begin{cor}\label{thm9}
Let $G$ be a $k$-distance regular graph with distance signless Laplacian spectrum $\{2k=\mu_{1}^{\mathcal{Q}}\geq\mu_{2}^{\mathcal{Q}}\geq\cdots\geq\mu_{n}^{\mathcal{Q}}\}$, then the distance signless Laplacian spectrum of $D_{2}G$ consists of eigenvalues $4k+4$, $2\mu_{i}^{\mathcal{Q}}+4$ for $i=2,3,\cdots,n$ and $2k$ with multiplicity $n$.
\end{cor}
\begin{proof}The proof follows from the above theorem by choosing $Tr_{G}(v_{i})=k$.\qed\end{proof}
\begin{cor}\label{thm10}
Let $G$ be a $k$-distance regular graph with distance signless Laplacian spectrum $\{2k=\mu_{1}^{\mathcal{Q}}\geq\mu_{2}^{\mathcal{Q}}\geq\cdots\geq\mu_{n}^{\mathcal{Q}}\}$, then the distance signless Laplacian energy of $D_{2}G$ is
\begin{equation*}
DSLE(D_{2}G)=2\sum_{i=2}^{\sigma}\mu_{i}^{\mathcal{Q}}-2\sum_{i=\sigma+1}^{n}\mu_{i}^{\mathcal{Q}}+2nk-4k\sigma+4k+4\sigma.
\end{equation*}
\end{cor}
\begin{proof}We can prove the theorem by using the similar method as for Corollary \ref{thm7}.\qed\end{proof}

\subsection{The $\mathcal{L}-$spectrum and $\mathcal{Q}-$spectrum of $G_{1}\nabla G_{2}$}

\begin{thm}\label{thm11}
For $i=1,2$, let $G_{i}$ be an $r_{i}$-regular graph with $n_{i}$ vertices and eigenvalues of the adjacency matrix $A(G_{i}), \{r_{1}=\lambda_{1}, \lambda_{2},\cdots,\lambda_{n_{1}}\}$ and $\{r_{2}=\mu_{1}, \mu_{2},\cdots,\mu_{n_{2}}\}$, respectively. The distance Laplacian spectrum of $G_{1}\nabla G_{2}$ consists of eigenvalues $0, n_{1}+n_{2}, 2n_{1}+n_{2}-r_{1}+\lambda_{j}$ for $j=2,3,\cdots,n_{1}$ and $2n_{2}+n_{1}-r_{2}+\mu_{j}$ for $j=2,3,\cdots,n_{2}$.
\end{thm}
\begin{proof} The distance Laplacian matrix of the join $G_{1}\nabla G_{2}$ has the form
\begin{equation} 
\mathcal{L}(G_{1}\nabla G_{2})=
\left[                 
  \begin{array}{cc}     
  \begin{smallmatrix}
    (2n_{1}+n_{2}-r_{1})I_{n_{1}}-2J_{n_{1}}+A(G_{1}) & -J_{n_{1}\times n_{2}}\\  
    -J_{n_{2}\times n_{1}}& (2n_{2}+n_{1}-r_{2})I_{n_{2}}-2J_{n_{2}}+A(G_{2})\\  
  \end{smallmatrix}
  \end{array}
\right].    \nonumber             
\end{equation}

\textbf{Claim 1.} As a regular graph, $G_{1}$ has the all-ones vector $J$ as an eigenvector corresponding to eigenvalue $r_{1}$, while all other eigenvectors are orthogonal to $J$. Let $\lambda_{j}$ be an arbitrary eigenvalue of the adjacency matrix of $G_{1}$ with corresponding eigenvector $x_{j}$ for $j=2,3,\cdots,n_{1}$, such that $J^{T}x_{j}=0$.

$\begin{matrix}
\Psi=
\left[                 
  \begin{array}{cc}   
    \begin{smallmatrix}
    x_{n_{1}\times1}  \\  
    0_{n_{2}\times1} \\  
    \end{smallmatrix}
  \end{array}
\right]    \nonumber             
\end{matrix}$
is an eigenvector of $\mathcal{L}(G_{1}\nabla G_{2})$ corresponding to the eigenvalue $2n_{1}+n_{2}-r_{1}+\lambda_{j}$ for $j=2,3,\cdots,n_{1}$.
\begin{equation*}
\begin{aligned} 
\mathcal{L}(G_{1}\nabla G_{2})\cdot\Psi&=
\left[                 
  \begin{array}{ccc}     
  \begin{smallmatrix}
    (2n_{1}+n_{2}-r_{1})I_{n_{1}}-2J_{n_{1}}+A(G_{1}) & -J_{n_{1}\times n_{2}}\\  
    -J_{n_{2}\times n_{1}}& (2n_{2}+n_{1}-r_{2})I_{n_{2}}-2J_{n_{2}}+A(G_{2})\\  
  \end{smallmatrix}
  \end{array}
\right]    \nonumber             
\begin{matrix}
\left[                 
  \begin{array}{ccc}   
    \begin{smallmatrix}
    x_{n_{1}\times1}  \\  
    0_{n_{2}\times1} \\  
    \end{smallmatrix}
  \end{array}
\right]    \nonumber             
\end{matrix},\\  &
\begin{matrix}
=
\left[                 
  \begin{array}{ccc}   
    \begin{smallmatrix}
     (2n_{1}+n_{2}-r_{1}+\lambda_{j})x_{n_{1}\times1}\\  
    0_{n_{2}\times1}\\  
    \end{smallmatrix}
  \end{array}
\right]    \nonumber             
\end{matrix},\\&
=(2n_{1}+n_{2}-r_{1}+\lambda_{j})\cdot\Psi.
\end{aligned}
\end{equation*}
Thus $2n_{1}+n_{2}-r_{1}+\lambda_{j}$ is an eigenvalue of $\mathcal{L}(G_{1}\nabla G_{2})$ for $j=2,3,\cdots,n_{1}$

\textbf{Claim 2.} As a regular graph, $G_{2}$ has the all-one vector $J$ as an eigenvector corresponding to eigenvalue $r_{2}$, while all other eigenvectors are orthogonal to $J$. Let $\mu_{j}$ be an arbitrary eigenvalue of the adjacency matrix of $G_{2}$ with corresponding eigenvector $y_{j}$ for $j=2,3,\cdots,n_{2}$, such that $J^{T}y_{j}=0$.

$\begin{matrix}
\Psi=
\left[                 
  \begin{array}{cc}   
    \begin{smallmatrix}
    0_{n_{1}\times1}  \\  
    y_{n_{2}\times1} \\  
    \end{smallmatrix}
  \end{array}
\right]    \nonumber             
\end{matrix}$
is an eigenvector of $\mathcal{L}(G_{1}\nabla G_{2})$ corresponding to the eigenvalue $2n_{2}+n_{1}-r_{2}+\mu_{j}$ for $j=2,3,\cdots,n_{2}$. For
\begin{equation*}
\begin{aligned} 
\mathcal{L}(G_{1}\nabla G_{2})\cdot\Psi&=
\left[                 
  \begin{array}{ccc}     
  \begin{smallmatrix}
    (2n_{1}+n_{2}-r_{1})I_{n_{1}}-2J_{n_{1}}+A(G_{1}) & -J_{n_{1}\times n_{2}}\\  
    -J_{n_{2}\times n_{1}}& (2n_{2}+n_{1}-r_{2})I_{n_{2}}-2J_{n_{2}}+A(G_{2})\\  
  \end{smallmatrix}
  \end{array}
\right]    \nonumber             
\begin{matrix}
\left[                 
  \begin{array}{ccc}   
    \begin{smallmatrix}
    0_{n_{1}\times1}  \\  
    y_{n_{2}\times1} \\  
    \end{smallmatrix}
  \end{array}
\right]    \nonumber             
\end{matrix},\\  &
\begin{matrix}
=
\left[                 
  \begin{array}{ccc}   
    \begin{smallmatrix}
    0_{n_{1}\times1}\\
     (2n_{2}+n_{1}-r_{2}+\mu_{j})y_{n_{2}\times1}\\  
    \end{smallmatrix}
  \end{array}
\right]    \nonumber             
\end{matrix},\\&
=(2n_{2}+n_{1}-r_{2}+\mu_{j})\cdot\Psi.
\end{aligned}
\end{equation*}
Thus $2n_{2}+n_{1}-r_{2}+\mu_{j}$ is an eigenvalue of $\mathcal{L}(G_{1}\nabla G_{2})$ for $j=2,3,\cdots,n_{2}$.

\textbf{Claim 3.} Suppose now that $v$ is an eigenvalue of $\mathcal{L}(G_{1}\nabla G_{2})$ with an eigenvector of the form
$\begin{matrix}
\Psi=
\left[                 
  \begin{array}{cc}   
    \begin{smallmatrix}
    \alpha J_{n_{1}\times1}  \\  
    \beta J_{n_{2}\times1} \\  
    \end{smallmatrix}
  \end{array}
\right]    \nonumber             
\end{matrix}$.
Then, from $\mathcal{L}(G_{1}\nabla G_{2})\Psi=v\Psi$, using $A(G_{1})J_{n_{1}\times1}=r_{1}J_{n_{1}\times1}$ and $A(G_{2})J_{n_{2}\times1}=r_{2}J_{n_{2}\times1}$, we get the following
\begin{equation*}
n_{2}\alpha-n_{2}\beta=v\alpha,
\end{equation*}
\begin{equation*}
n_{1}\beta-n_{1}\alpha=v\beta,
\end{equation*}
Eliminating $\alpha$ and $\beta$, we get the quadratic equation in $v$
\begin{equation*}
v^{2}-(n_{1}+n_{2})v=0.
\end{equation*}
Thus $0$ and $n_{1}+n_{2}$ is the eigenvalue of $\mathcal{L}(G_{1}\nabla G_{2})$.\qed\end{proof}

Note that the complete bipartite graph $K_{n_{1},n_{2}}$ is isomorphic to a join $\overline{K}_{n_{1}}\nabla\overline{K}_{n_{2}}$ of the empty graphs $\overline{K}_{n_{1}}$ and  $\overline{K}_{n_{2}}$. Hence
\begin{cor}\label{thm12}
The distance Laplacian spectrum of the complete bipartite graph $K_{n_{1},n_{2}}$ consists of simple eigenvalues $0$, $n_{1}+n_{2}$, $2n_{1}+n_{2}$ with multiplicity $n_{1}-1$ and $n_{1}+2n_{2}$ with multiplicity $n_{2}-1$.
\end{cor}

\begin{thm}\label{thm13}
For $i=1,2$, let $G_{i}$ be an $r_{i}$-regular graph with $n_{i}$ vertices and eigenvalues of the adjacency matrix $A(G_{i}), \{r_{1}=\lambda_{1}, \lambda_{2},\cdots,\lambda_{n_{1}}\}$ and $\{r_{2}=\mu_{1}, \mu_{2},\cdots,\mu_{n_{2}}\}$, respectively. The distance signless Laplacian spectrum of $G_{1}\nabla G_{2}$ consists of eigenvalues $2n_{1}+n_{2}-r_{1}-\lambda_{j}-4$ for $j=2,3,\cdots,n_{1}$, $2n_{2}+n_{1}-r_{2}-\mu_{j}-4$ for $j=2,3,\cdots,n_{2}$ and two eigenvalues of the form
\begin{equation*}
\frac{5}{2}(n_{1}+n_{2})-r_{1}-r_{2}-4\pm\frac{\sqrt{\bigtriangleup}}{2}
\end{equation*}where $\bigtriangleup=9n_{1}^{2}-14n_{1}n_{2}-12n_{1}r_{1}+12n_{1}r_{2}+9n_{2}^{2}+12n_{2}r_{1}-12n_{2}r_{2}+4r_{1}^{2}-8r_{1}r_{2}+4r_{2}^{2}$.
\end{thm}
\begin{proof} The distance signless Laplacian matrix of the join $G_{1}\nabla G_{2}$ has the form
\begin{equation} 
\mathcal{Q}(G_{1}\nabla G_{2})=
\left[                 
  \begin{array}{cc}     
  \begin{smallmatrix}
    (2n_{1}+n_{2}-r_{1}-4)I_{n_{1}}+2J_{n_{1}}-A(G_{1}) & J_{n_{1}\times n_{2}}\\  
    J_{n_{2}\times n_{1}}& (2n_{2}+n_{1}-r_{2}-4)I_{n_{2}}+2J_{n_{2}}-A(G_{2})\\  
  \end{smallmatrix}
  \end{array}
\right],    \nonumber             
\end{equation}
we can prove the theorem by using the similar method as for Theorem \ref{thm11}.\qed\end{proof}

Note that the complete bipartite graph $K_{n_{1},n_{2}}$ is isomorphic to a join $\overline{K}_{n_{1}}\nabla\overline{K}_{n_{2}}$ of the empty graphs $\overline{K}_{n_{1}}$ and  $\overline{K}_{n_{2}}$. Hence
\begin{cor}\label{thm14}
The distance signless Laplacian spectrum of the complete bipartite graph $K_{n_{1},n_{2}}$ consists of simple eigenvalues $\frac{5(n_{1}+n_{2})-8\pm\sqrt{9(n_{1}-n_{2})^{2}+4n_{1}n_{2}}}{2}$, $2n_{1}+n_{2}-4$ with multiplicity $n_{1}-1$ and $n_{1}+2n_{2}-4$ with multiplicity $n_{2}-1$.
\end{cor}

\subsection{The $\mathcal{L}-$spectrum and $\mathcal{Q}-$spectrum of $G_{1}\nabla (G_{2}\cup G_{3})$}
\begin{thm}\label{thm15}
For $i=1,2,3$, let $G_{i}$ be an $r_{i}$-regular graph with $n_{i}$ vertices and eigenvalues of the adjacency matrix $A(G_{i}), \{r_{1}=\lambda_{1}, \lambda_{2},\cdots,\lambda_{n_{1}}\}$, $\{r_{2}=\mu_{1}, \mu_{2},\cdots,\mu_{n_{2}}\}$ and $\{r_{3}=\delta_{1}, \delta_{2},\cdots,\delta_{n_{3}}\}$, respectively. The distance Laplacian spectrum of $G_{1}\nabla (G_{2}\cup G_{3})$ consists of eigenvalues $0$,  $n_{1}+n_{2}+n_{3}$, $n_{1}+2n_{2}+2n_{3}$, $2n_{1}+n_{2}+n_{3}-r_{1}+\lambda_{j}$ for $j=2,3,\cdots,n_{1}$, $n_{1}+2n_{2}+2n_{3}-r_{2}+\mu_{j}$ for $j=2,3,\cdots,n_{2}$ and $n_{1}+2n_{2}+2n_{3}-r_{3}+\delta_{j}$ for $j=2,3,\cdots,n_{3}$.
\end{thm}
\begin{proof} The distance Laplacian matrix of $G_{1}\nabla (G_{2}\cup G_{3})$ has the form
\begin{equation} 
\mathcal{L}(G_{1}\nabla (G_{2}\cup G_{3}))=
\left[                 
  \begin{array}{cc}     
  \begin{smallmatrix}
    (2n_{1}+n_{2}+n_{3}-r_{1})I_{n_{1}}-2J_{n_{1}}+A(G_{1}) & -J_{n_{1}\times n_{2}} & -J_{n_{1}\times n_{3}}\\  
   -J_{n_{2}\times n_{1}}& (n_{1}+2n_{2}+2n_{3}-r_{2})I_{n_{2}}-2J_{n_{2}}+A(G_{2})& -2J_{n_{2}\times n_{3}}\\  
   -J_{n_{3}\times n_{1}}&  -2J_{n_{3}\times n_{2}} & L^{\ast}\\  
  \end{smallmatrix}
  \end{array}
\right],    \nonumber             
\end{equation}
where $L^{\ast}=(n_{1}+2n_{2}+2n_{3}-r_{3})I_{n_{3}}-2J_{n_{3}}+A(G_{3})$,
and we can prove the theorem by using the similar method as for Theorem \ref{thm11}.\qed\end{proof}

\begin{thm}\label{thm16}
For $i=1,2,3$, let $G_{i}$ be an $r_{i}$-regular graph with $n_{i}$ vertices and eigenvalues of the adjacency matrix $A
(G_{i}), \{r_{1}=\lambda_{1}, \lambda_{2},\cdots,\lambda_{n_{1}}\}$, $\{r_{2}=\mu_{1}, \mu_{2},\cdots,\mu_{n_{2}}\}$ and $\{r_{3}=\delta_{1}, \delta_{2},\cdots,\delta_{n_{3}}\}$, respectively. The distance signless Laplacian spectrum of $G_{1}\nabla (G_{2}\cup G_{3})$ consists of eigenvalues $2n_{1}+n_{2}+n_{3}-r_{1}-4-\lambda_{j}-4$ for $j=2,3,\cdots,n_{1}$, $n_{1}+2n_{2}+2n_{3}-r_{2}-\mu_{j}-4$ for $j=2,3,\cdots,n_{2}$, $n_{1}+2n_{2}+2n_{3}-r_{3}-\delta_{j}-4$ for $j=2,3,\cdots,n_{3}$ and the three roots of the equation
\begin{equation*}
x^{3}+Ax^{2}+Bx+C=0
\end{equation*}
where $A=-6n_{1}-7n_{2}-7n_{3}+2r_{1}+2r_{2}+2r_{3}+12$,  $B=-48n_{1} - 56n_{2} - 56n_{3} + 16r_{1} + 16r_{2} + 16r_{3 }+ 31n_{1}n_{2} + 31n_{1}n_{3} + 28n_{2}n_{3} - 4n_{1}r_{1} - 10n_{1}r_{2} - 12n_{2}r_{1} - 10n_{1}r_{3} - 6n_{2}r_{2} - 12n_{3}r_{1} - 10n_{2}r_{3} - 10n_{3}r_{2} - 6n_{3}r_{3} + 4r_{1}r_{2} + 4r_{1}r_{3} + 4r_{2}r_{3} + 9n_{1}^2 + 14n_{2}^2 + 14n_{3}^2 + 48$ and $C=8r_{1}r_{2}r_{3} - 16n_{3}r_{1}r_{2} - 16n_{2}r_{1}r_{3} - 16n_{1}r_{2}r_{3} - 8n_{3}r_{1}r_{3} - 8n_{2}r_{1}r_{2} - 4n_{3}r_{2}r_{3} - 4n_{2}r_{2}r_{3} - 4n_{1}r_{1}r_{3} - 4n_{1}r_{1}r_{2} + 32n_{2}n_{3}r_{1} + 32n_{1}n_{3}r_{2} + 32n_{1}n_{2}r_{3} + 18n_{1}n_{3}r_{3 }+ 18n_{1}n_{2}r_{2} + 12n_{2}n_{3}r_{3} + 12n_{2}n_{3}r_{2} + 12n_{1}n_{3}r_{1} + 12n_{1}n_{2}r_{1} - 72n_{1}n_{2}n_{3} + 16n_{3}^2r_{1} + 16n_{2}^2r_{1} + 8n_{3}^2r_{2} + 8n_{2}^2r_{3} + 8n_{1}^2r_{3} + 8n_{1}^2r_{2} + 4n_{3}^2r_{3} + 4n_{2}^2r_{2} + 2n_{1}^2r_{1} - 36n_{1}n_{3}^2 - 36n_{1}n_{2}^2 - 24n_{2}^2n_{3} - 24n_{2}n_{3}^2 - 24n_{1}^2n_{3 }- 24n_{1}^2n_{2} + 16r_{2}r_{3} + 16r_{1}r_{3} + 16r_{1}r_{2} - 48n_{3}r_{1} - 48n_{2}r_{1} - 40n_{3}r_{2} - 40n_{2}r_{3} - 40n_{1}r_{3} - 40n_{1}r_{2} - 24n_{3}r_{3} - 24n_{2}r_{2} - 16n_{1}r_{1} + 124n_{1}n_{3} + 124n_{1}n_{2} + 112n_{2}n_{3} + 32r_{3} + 32r_{2} + 32r_{1} - 112n_{3} - 112n_{2} - 96n_{1} + 56n_{3}^2 + 56n_{2}^2 + 36n_{1}^2 - 8n_{3}^3 - 8n_{2}^3 - 4n_{1}^3 + 64$.
\end{thm}
\begin{proof} The distance signless Laplacian matrix of $G_{1}\nabla (G_{2}\cup G_{3})$ has the form
\begin{equation} 
\mathcal{Q}(G_{1}\nabla (G_{2}\cup G_{3}))=
\left[                 
  \begin{array}{cc}     
  \begin{smallmatrix}
    (2n_{1}+n_{2}+n_{3}-r_{1}-4)I_{n_{1}}+2J_{n_{1}}-A(G_{1}) & J_{n_{1}\times n_{2}} & J_{n_{1}\times n_{3}}\\  
    J_{n_{2}\times n_{1}}& (n_{1}+2n_{2}+2n_{3}-r_{2}-4)I_{n_{2}}+2J_{n_{2}}-A(G_{2})& 2J_{n_{2}\times n_{3}}\\  
    J_{n_{3}\times n_{1}}&  2J_{n_{3}\times n_{2}} & Q^{\ast}\\  
  \end{smallmatrix}
  \end{array}
\right],    \nonumber             
\end{equation}
where $Q^{\ast}=(n_{1}+2n_{2}+2n_{3}-r_{3}-4)I_{n_{3}}+2J_{n_{3}}-A(G_{3})$, and we can prove the theorem by using the similar method as for Theorem \ref{thm11}.\qed\end{proof}

\subsection{The $\mathcal{D}-$spectrum, $\mathcal{L}-$spectrum and $\mathcal{Q}-$spectrum of $G_{1}\oplus G_{2}$}

\begin{thm}\label{thm17}
Let $G_{i}$ be an $r_{i}-$regular graph of order $p_{i}$ with the adjacency matrix $A_{i}$ and adjacency spectrum $\{r_{i},\lambda_{i_{2}},\lambda_{i_{3}},\cdots,\lambda_{i_{p_{i}}}\}$ for $i=1,2$. Then, the distance spectrum of $G_{1}\oplus G_{2}$ is:\\
$(1)$ $-(\lambda_{1_{j}}+3\pm\sqrt{(\lambda_{1_{j}}+1)^{2}+4(\lambda_{1_{j}}+r_{1})})$, $j=2,3,\cdots,p_{1}$;\\
$(2)$ $-(\lambda_{2_{j}}+3\pm\sqrt{(\lambda_{2_{j}}+1)^{2}+4(\lambda_{2_{j}}+r_{2})})$, $j=2,3,\cdots,p_{2}$;\\
$(3)$ $-2$ with multiplicity $\frac{1}{2}p_{1}r_{1}+\frac{1}{2}p_{2}r_{2}-p_{1}-p_{2}$; \\
$(4)$ together with the four eigenvalues of
\begin{equation} 
\left[                 
  \begin{array}{cccc}     
    4p_{1}-2r_{1}-4 & \frac{3}{2}p_{1}r_{1}-2r_{1} & 3p_{2} & p_{2}r_{2} \\  
    3p_{1}-4 & p_{1}r_{1}-2 &  2p_{2}&  \frac{1}{2}p_{2}r_{2}\\  
    3p_{1} & p_{1}r_{1} & 4p_{2}-2r_{2}-4&  \frac{3}{2}p_{2}r_{2}-2r_{2}\\
    2p_{1} & \frac{1}{2}p_{1}r_{1} & 3p_{2}-4&  p_{2}r_{2}-2
  \end{array}
\right].    \nonumber             
\end{equation}
\end{thm}

\begin{proof} Let $R_{i}$ be the incidence matrix of $G_{i}$ for $i=1,2$. Then, by a proper ordering of the vertices of $G_{1}\oplus G_{2}$,
its distance matrix $\mathcal{D}(G_{1}\oplus G_{2})$ can be written in the form
\begin{equation} 
\mathcal{D}(G_{1}\oplus G_{2})=
\left[                 
  \begin{array}{cccc}     
    4(J-I)-2A_{1} & 3J-2R_{1} & 3J & 2J\\  
    3J-2R_{1}^{T} & 2(J-I) &  2J &  J\\  
    3J & 2J & 4(J-I)-2A_{2} &  3J-2R_{2}\\
    2J & J & 3J-2R_{2}^{T} & 2(J-I)
  \end{array}
\right],    \nonumber             
\end{equation}
where $J$ is the all-one matrix, and $I$ is the identity matrix of appropriate orders.

Let $\lambda_{i_{j}}\neq r_{i}$ be an eigenvalue of $A_{i}$ with an eigenvector $X_{i_{j}}$ for $i=1,2$ and $j=2,3,\cdots,p_{i}$. Then, $X_{i_{j}}$ is orthogonal to the all-one matrix $J$, and $A_{i}X_{i_{j}}=\lambda_{i_{j}} X_{i_{j}}$. Now, by Lemma \ref{lem10}, we have\\
\begin{equation*}
\begin{aligned} 
R_{i}R_{i}^{T}=&A_{i}+r_{i}I,\\
R_{i}R_{i}^{T}X_{i_{j}}=&(A_{i}+r_{i}I)X_{i_{j}}\\
      =&(\lambda_{i_{j}}+r_{i})X_{i_{j}},\\
      A(L(G_{i}))=&R_{i}^{T}R_{i}-2I,\\
      A(L(G_{i}))R_{i}^{T}X_{i_{j}}=&(R_{i}^{T}R_{i}-2I)R_{i}^{T}X_{i_{j}}\\
      =&R_{i}^{T}(A_{i}+r_{i}I)X_{i_{j}}-2R_{i}^{T}X_{i_{j}}\\
      =&(\lambda_{i_{j}}+r_{i}-2)R_{i}^{T}X_{i_{j}}.
\end{aligned}
\end{equation*}
Therefore, $R_{i}^{T}X_{i_{j}}$ is an eigenvector of $A(L(G_{i}))$ with an eigenvalue $\lambda_{i_{j}}+r_{i}-2$, which is different from its regularity as $\lambda_{i_{j}}\neq r_{i}$. Then, $R_{i}^{T}X_{i_{j}}$ is orthogonal to the all-one vector.

Now, consider the vector
$\begin{matrix}
\phi_{1}=
\left[                 
  \begin{array}{cccc}   
    tX_{1_{j}}  \\  
    R_{1}^{T}X_{1_{j}} \\  
     0\\
     0
  \end{array}
\right]    \nonumber             
\end{matrix}$,
which
is an eigenvector of $\mathcal{D}(G_{1}\oplus G_{2})$ with eigenvalues $\mu_{1}$, then from the equation $\mathcal{D}(G_{1}\oplus G_{2})\phi_{1}=\mu_{1}\phi_{1}$, we get
\begin{equation}\label{LGQQEqu1}
\begin{aligned} 
-(2\lambda_{1_{j}}+4)t-2(\lambda_{1_{j}}+r_{1})=\mu_{1}t,
\end{aligned}
\end{equation}
\begin{equation}\label{LGQQEqu2}
\begin{aligned} 
-2t-2=\mu_{1}.
\end{aligned}
\end{equation}
Thus, solving $Eqs.(\ref{LGQQEqu1})$ and $Eqs.(\ref{LGQQEqu2})$ yields $\mu_{1}=-(\lambda_{1_{j}}+3\pm\sqrt{(\lambda_{1_{j}}+1)^{2}+4(\lambda_{1_{j}}+r_{1})})$, $j=2,3,\cdots,p_{1}$.

Then, consider the vector
$\begin{matrix}
\phi_{2}=
\left[                 
  \begin{array}{cccc}   
    0  \\  
    0\\  
     tX_{2_{j}}\\
      R_{2}^{T}X_{2_{j}}
  \end{array}
\right]    \nonumber             
\end{matrix}$, which
is an eigenvector of $\mathcal{D}(G_{1}\oplus G_{2})$ with eigenvalues $\mu_{2}$, then from the equation $\mathcal{D}(G_{1}\oplus G_{2})\phi_{2}=\mu_{2}\phi_{2}$, we get
\begin{equation}\label{LGQQEqu3}
\begin{aligned} 
-(2\lambda_{2_{j}}+4)t-2(\lambda_{2_{j}}+r_{2})=\mu_{2}t,
\end{aligned}
\end{equation}
\begin{equation}\label{LGQQEqu4}
\begin{aligned} 
-2t-2=\mu_{2}.
\end{aligned}
\end{equation}
Thus, solving $Eqs.(\ref{LGQQEqu3})$ and $Eqs.(\ref{LGQQEqu4})$ yields $\mu_{2}=-(\lambda_{2_{j}}+3\pm\sqrt{(\lambda_{2_{j}}+1)^{2}+4(\lambda_{2_{j}}+r_{2})})$, $j=2,3,\cdots,p_{2}$.

Let $Z_{i}$ be an eigenvector of $L(G_{i})$ with the eigenvalue $-2$. Then, by Lemma \ref{lem11}, $R_{i}Z_{i}=0$. Now, let
$\begin{matrix}
\psi_{1}=
\left[                 
  \begin{array}{cccc}   
    0  \\  
    Z_{1} \\  
     0\\
     0
  \end{array}
\right]    \nonumber             
\end{matrix}$ and
$\begin{matrix}
\psi_{2}=
\left[                 
  \begin{array}{cccc}   
    0  \\  
    0 \\  
     0\\
     Z_{2}
  \end{array}
\right]    \nonumber             
\end{matrix}$
is an eigenvector of $\mathcal{D}(G_{1}\oplus G_{2})$ corresponding to the eigenvalue $-2$ with multiplicity $\frac{1}{2}p_{i}r_{i}-p_{i}$.
 \begin{equation*}
\begin{aligned} 
\mathcal{D}(G_{1}\oplus G_{2})\psi_{i}&=
\left[                 
  \begin{array}{cccc}     
    4(J-I)-2A_{1} & 3J-2R_{1} & 3J & 2J\\  
    3J-2R_{1}^{T} & 2(J-I) &  2J &  J\\  
    3J & 2J & 4(J-I)-2A_{2} &  3J-2R_{2}\\
    2J & J & 3J-2R_{2}^{T} & 2(J-I)
  \end{array}
\right]    \nonumber             
\psi_{i},\\  &
=-2\psi_{i}.
\end{aligned}
\end{equation*}

Thus, forming eigenvectors of this type we get $p_{1}+p_{2}+\frac{1}{2}p_{1}r_{1}+\frac{1}{2}p_{2}r_{2}-4$ eigenvectors, and there remains 4. By
the construction, the remaining four are of the form
$\begin{matrix}
v=
\left[                 
  \begin{array}{cccc}   
    \alpha J  \\  
    \beta J \\  
     \gamma J\\
     \delta J
  \end{array}
\right]    \nonumber             
\end{matrix}$
for some $(\alpha,\beta,\gamma,\delta)\neq(0,0,0,0)$. If $\sigma$ be an eigenvalue of $\mathcal{D}(G_{1}\oplus G_{2})$ with an eigenvector $v$,
then from $\mathcal{D}(G_{1}\oplus G_{2})v=\sigma v$ we can see that the remaining four are the eigenvalues of the matrix
\begin{equation} 
\left[                 
  \begin{array}{cccc}     
    4p_{1}-2r_{1}-4 & \frac{3}{2}p_{1}r_{1}-2r_{1} & 3p_{2} & p_{2}r_{2} \\  
    3p_{1}-4 & p_{1}r_{1}-2 &  2p_{2}&  \frac{1}{2}p_{2}r_{2}\\  
    3p_{1} & p_{1}r_{1} & 4p_{2}-2r_{2}-4&  \frac{3}{2}p_{2}r_{2}-2r_{2}\\
    2p_{1} & \frac{1}{2}p_{1}r_{1} & 3p_{2}-4&  p_{2}r_{2}-2
  \end{array}
\right].    \nonumber             
\end{equation}
This completes the proof.\qed\end{proof}

\begin{thm}\label{thm18}
Let $G_{i}$ be an $r_{i}-$regular graph of order $p_{i}$ with the adjacency matrix $A_{i}$ and adjacency spectrum $\{r_{i},\lambda_{i_{2}},\lambda_{i_{3}},\cdots,\lambda_{i_{p_{i}}}\}$ for $i=1,2$. Then, the distance Laplacian spectrum of $G_{1}\oplus G_{2}$ is:\\
$(1)$ $\lambda_{1_{j}}+\frac{7p_{1}}{2}+\frac{5p_{2}}{2}+\frac{5p_{1}r_{1}}{4}+\frac{3p_{2}r_{2}}{4}-2r_{1}-2\pm\frac{\sqrt{4\lambda_{1_{j}}^2 + 2\lambda_{1_{j}}p_{1}r_{1} + 4\lambda_{1_{j}}p_{1} + 2\lambda_{1_{j}}p_{2}r_{2} + 4\lambda_{1_{j}}p_{2} - 16\lambda_{1_{j}}r_{1} + 32\lambda_{1_{j}} + \frac{1}{4}p_{1}^2r_{1}^2}}{2}\\  \frac{\overline{ + p_{1}^2r_{1}+ p_{1}^2 + \frac{1}{2}p_{1}p_{2}r_{1}r_{2} + p_{1}p_{2}r_{1} + p_{1}p_{2}r_{2} + 2p_{1}p_{2} - 4p_{1}r_{1}^2 - 4p_{1}r_{1} + 8p_{1} + \frac{1}{4}p_{2}^2r_{2}^2 + p_{2}^2r_{2} + p_{2}^2 - 4p_{2}r_{1}r_{2} - 8p_{2}r_{1} + 4p_{2}r_{2} + 8p_{2} + 16r_{1}^2 }}{2}\\
\frac{\overline{ - 16r_{1} + 16}}{2}$, $j=2,3,\cdots,p_{1}$;\\
$(2)$ $\lambda_{2_{j}}+\frac{5p_{1}}{2}+\frac{7p_{2}}{2}+\frac{3p_{1}r_{1}}{4}+\frac{5p_{2}r_{2}}{4}-2r_{2}-2\pm\frac{\sqrt{4\lambda_{2_{j}}^2 + 2\lambda_{2_{j}}p_{1}r_{1} + 4\lambda_{2_{j}}p_{1} + 2\lambda_{2_{j}}p_{2}r_{2} + 4\lambda_{2_{j}}p_{2} - 16\lambda_{2_{j}}r_{2} + 32\lambda_{2_{j}} + \frac{1}{4}p_{1}^2r_{1}^2 }}{2}\\  \frac{\overline{+ p_{1}^2r_{1} + p_{1}^2 + \frac{1}{2}p_{1}p_{2}r_{1}r_{2} + p_{1}p_{2}r_{1} + p_{1}p_{2}r_{2} + 2p_{1}p_{2} - 4p_{1}r_{1}r_{2} + 4p_{1}r_{1} -8p_{1}r_{2} + 8p_{1}+ \frac{1}{4}p_{2}^2r_{2}^2 + p_{2}^2r_{2} + p_{2}^2 - 4p_{2}r_{2}^{2}-4p_{2}r_{2} + 8p_{2}  + 16r_{2}^2}}{2}\\
\frac{\overline{ - 16r_{2} + 16}}{2}$, $j=2,3,\cdots,p_{2}$;\\
$(3)$ $3p_{1}+p_{1}r_{1}+2p_{2}+\frac{1}{2}p_{2}r_{2}-4$ with multiplicity $\frac{1}{2}p_{1}r_{1}-p_{1}$;\\
$(4)$ $2p_{1}+\frac{1}{2}p_{1}r_{1}+3p_{2}+p_{2}r_{2}-4$ with multiplicity $\frac{1}{2}p_{2}r_{2}-p_{2}$;\\
$(5)$ together with the four eigenvalues of
\begin{equation} 
\left[                 
  \begin{array}{cccc}     
  \begin{smallmatrix}
    \frac{3}{2}p_{1}r_{1}+3p_{2}+p_{2}r_{2}-2r_{1}& 2r_{1}-\frac{3}{2}p_{1}r_{1} & -3p_{2} & -p_{2}r_{2} \\  
    4-3p_{1} & 3p_{1}+2p_{2}+\frac{1}{2}p_{2}r_{2}-4 &  -2p_{2}&  -\frac{1}{2}p_{2}r_{2}\\  
    -3p_{1} & -p_{1}r_{1} & 3p_{1}+p_{1}r_{1}+\frac{3}{2}p_{2}r_{2}-2r_{2}&  2r_{2}-\frac{3}{2}p_{2}r_{2}\\
    -2p_{1} & -\frac{1}{2}p_{1}r_{1} & 4-3p_{2}& 2p_{1}+\frac{1}{2}p_{1}r_{1}+3p_{2}-4
  \end{smallmatrix}
  \end{array}
\right].    \nonumber             
\end{equation}
\end{thm}

\begin{proof} Let $R_{i}$ be the incidence matrix of $G_{i}$ for $i=1,2$. Then, by a proper ordering of the vertices of $G_{1}\oplus G_{2}$,
its distance Laplacian matrix $\mathcal{L}(G_{1}\oplus G_{2})$ can be written in the form
$$\scriptsize{
\begin{aligned}
\mathcal{L}(G_{1}\oplus G_{2})=
\left[                 
  \begin{array}{cccc}     
  \begin{smallmatrix}
    (4p_{1}+\frac{3}{2}p_{1}r_{1}+3p_{2}+p_{2}r_{2}-4r_{1})I-4J+2A_{1} & 2R_{1}-3J & -3J & -2J\\  
    2R_{1}^{T}-3J & (3p_{1}+p_{1}r_{1}+2p_{2}+\frac{1}{2}p_{2}r_{2}-4)I-2J &  -2J &  -J\\  
    -3J & -2J & B &  2R_{2}-3J\\
    -2J & -J & 2R_{2}^{T}-3J & C
  \end{smallmatrix}
  \end{array}
\right],    \nonumber             
\end{aligned}}$$
where $B=(3p_{1}+p_{1}r_{1}+4p_{2}+\frac{3}{2}p_{2}r_{2}-4r_{2})I-4J+2A_{2}$, $C=(2p_{1}+\frac{1}{2}p_{1}r_{1}+3p_{2}+p_{2}r_{2}-4)I-2J$, $J$ is the all-one matrix, and $I$ is the identity matrix of appropriate orders. We can prove the theorem in a similar way to Theorem
\ref{thm17}.\qed\end{proof}

\begin{thm}\label{thm19}
Let $G_{i}$ be an $r_{i}-$regular graph of order $p_{i}$ with the adjacency matrix $A_{i}$ and adjacency spectrum $\{r_{i},\lambda_{i_{2}},\lambda_{i_{3}},\cdots,\lambda_{i_{p_{i}}}\}$ for $i=1,2$. Then, the distance signless Laplacian spectrum of $G_{1}\oplus G_{2}$ is:\\
$(1)$ $-\lambda_{1_{j}}+\frac{7p_{1}}{2}+\frac{5p_{2}}{2}+\frac{5p_{1}r_{1}}{4}+\frac{3p_{2}r_{2}}{4}-2r_{1}-8\pm\frac{\sqrt{4\lambda_{1_{j}}^2 - 2\lambda_{1_{j}}p_{1}r_{1} - 4\lambda_{1_{j}}p_{1} - 2\lambda_{1_{j}}p_{2}r_{2} - 4\lambda_{1_{j}}p_{2} + 16\lambda_{1_{j}}r_{1} + 16\lambda_{1_{j}} + }}{2}\\
 \frac{\overline{\frac{1}{4}p_{1}^2r_{1}^2 +p_{1}^2r_{1} + p_{1}^2 + \frac{1}{2}p_{1}p_{2}r_{1}r_{2} + p_{1}p_{2}r_{1} + p_{1}p_{2}r_{2} + 2p_{1}p_{2} - 4p_{1}r_{1}^2 - 8p_{1}r_{1} + \frac{1}{4}p_{2}^2r_{2}^2 + p_{2}^2r_{2} + p_{2}^2 - 4p_{2}r_{1}r_{2} - 8p_{2}r_{1} + 16r_{1}^2 + 16r_{1}}}{2}$, $j=2,3,\cdots,p_{1}$;\\
$(2)$ $-\lambda_{2_{j}}+\frac{5p_{1}}{2}+\frac{7p_{2}}{2}+\frac{3p_{1}r_{1}}{4}+\frac{5p_{2}r_{2}}{4}-2r_{2}-8\pm\frac{\sqrt{4\lambda_{2_{j}}^2 - 2\lambda_{2_{j}}p_{1}r_{1} - 4\lambda_{2_{j}}p_{1} - 2\lambda_{2_{j}}p_{2}r_{2} - 4\lambda_{2_{j}}p_{2} + 16\lambda_{2_{j}}r_{2} + 16\lambda_{2_{j}} + }}{2}\\
 \frac{\overline{\frac{1}{4}p_{1}^2r_{1}^2 +p_{1}^2r_{1} + p_{1}^2 + \frac{1}{2}p_{1}p_{2}r_{1}r_{2} + p_{1}p_{2}r_{1} + p_{1}p_{2}r_{2} + 2p_{1}p_{2} - 4p_{1}r_{1}r_{2} - 8p_{1}r_{2} + \frac{1}{4}p_{2}^2r_{2}^2 + p_{2}^2r_{2} + p_{2}^2 - 4p_{2}r_{2}^2 - 8p_{2}r_{2} + 16r_{2}^2 + 16r_{2}}}{2}$, $j=2,3,\cdots,p_{2}$;\\
$(3)$ $3p_{1}+p_{1}r_{1}+2p_{2}+\frac{1}{2}p_{2}r_{2}-8$ with multiplicity $\frac{1}{2}p_{1}r_{1}-p_{1}$;\\
$(4)$ $2p_{1}+\frac{1}{2}p_{1}r_{1}+3p_{2}+p_{2}r_{2}-8$ with multiplicity $\frac{1}{2}p_{2}r_{2}-p_{2}$;\\
$(5)$ together with the four eigenvalues of
$$\scriptsize{
\begin{aligned}
\left[                 
  \begin{array}{cccc}     
  \begin{smallmatrix}
    8p_{1}+\frac{3}{2}p_{1}r_{1}+3p_{2}+p_{2}r_{2}-6r_{1}-8& \frac{3}{2}p_{1}r_{1}-2r_{1} & 3p_{2} & p_{2}r_{2} \\  
    3p_{1}-4 & 3p_{1}+2p_{1}r_{1}+2p_{2}+\frac{1}{2}p_{2}r_{2}-8 &  2p_{2}&  \frac{1}{2}p_{2}r_{2}\\  
    3p_{1} & p_{1}r_{1} & 3p_{1}+p_{1}r_{1}+8p_{2}+\frac{3}{2}p_{2}r_{2}-6r_{2}-8&  \frac{3}{2}p_{2}r_{2}-2r_{2}\\
    2p_{1} & \frac{1}{2}p_{1}r_{1} & 3p_{2}-4& 2p_{1}+\frac{1}{2}p_{1}r_{1}+3p_{2}+2p_{2}r_{2}-8
  \end{smallmatrix}
  \end{array}
\right].    \nonumber             
\end{aligned}}$$
\end{thm}

\begin{proof} Let $R_{i}$ be the incidence matrix of $G_{i}$ for $i=1,2$. Then, by a proper ordering of the vertices of $G_{1}\oplus G_{2}$,
its distance signless Laplacian matrix $\mathcal{Q}(G_{1}\oplus G_{2})$ can be written in the form\\
$$\scriptsize{
\begin{aligned}
\mathcal{Q}(G_{1}\oplus G_{2})=&
\left[                 
  \begin{array}{cccc}     
  \begin{smallmatrix}
    (4p_{1}+\frac{3}{2}p_{1}r_{1}+3p_{2}+p_{2}r_{2}-4r_{1}-8)I+4J-2A_{1}&3J-2R_{1}&3J&2J\\  
    3J-2R_{1}^{T}&(3p_{1}+p_{1}r_{1}+2p_{2}+\frac{1}{2}p_{2}r_{2}-8)I+2J&2J&J\\  
    3J&2J&B&3J-2R_{2}\\
    2J&J&3J-2R_{2}^{T}&C
  \end{smallmatrix}
  \end{array}
\right],    \nonumber             
\end{aligned}}$$\\
where $B=(3p_{1}+p_{1}r_{1}+4p_{2}+\frac{3}{2}p_{2}r_{2}-4r_{2}-8)I+4J-2A_{2}$, $C=(2p_{1}+\frac{1}{2}p_{1}r_{1}+3p_{2}+p_{2}r_{2}-8)I+2J$, $J$ is the all-one matrix and $I$ is the identity matrix of appropriate orders. We can prove the theorem in a similar way to Theorem
 \ref{thm17}.\qed\end{proof}
\subsection{The $\mathcal{D}-$spectrum, $\mathcal{L}-$spectrum and $\mathcal{Q}-$spectrum of $G_{1}\dot{\vee}G_{2}$}
\begin{thm}\label{thm20}
Let $G_{i}$ be an $r_{i}-$regular graph of order $p_{i}$ with the adjacency matrix $A_{i}$ and adjacency spectrum $\{r_{i},\lambda_{i_{2}},\lambda_{i_{3}},\cdots,\lambda_{i_{p_{i}}}\}$ for $i=1,2$. Then, the distance spectrum of $G_{1}\dot{\vee}G_{2}$ is:\\
$(1)$ $-2(\lambda_{1_{j}}+r_{1}+1)$, $j=2,3,\cdots,p_{1}$;\\
$(2)$ $-2(\lambda_{2_{j}}+r_{2}+1)$, $j=2,3,\cdots,p_{2}$;\\
$(3)$ $0$ with multiplicity $\frac{1}{2}r_{1}p_{1}+\frac{1}{2}r_{2}p_{2}-2$;\\
$(4)$ together with the four eigenvalues of
\begin{equation} 
\left[                 
  \begin{array}{cccc}     
    2p_{1}-2 & \frac{3}{2}r_{1}p_{1}-2r_{1} & p_{2}& r_{2}p_{2} \\  
    3p_{1}-4 & 2r_{1}p_{1}-4r_{1} &  2p_{2}& \frac{3}{2}r_{2}p_{2}\\  
    p_{1} & r_{1}p_{1} & 2p_{2}-2 & \frac{3}{2}r_{2}p_{2}-2r_{2}\\
    2p_{1} & \frac{3}{2}r_{1}p_{1} & 3p_{2}-4 & 2r_{2}p_{2}-4r_{2}
  \end{array}
\right].    \nonumber             
\end{equation}
\end{thm}
\begin{proof} Let $R_{i}$ be the incidence matrix of $G_{i}$ and $A(L(G_{i}))$ be the adjacency matrix of $L(G_{i})$, Then, by a proper ordering of the vertices of $G_{1}\dot{\vee}G_{2}$,
its distance matrix $\mathcal{D}(G_{1}\dot{\vee}G_{2})$ can be written in the form
\begin{equation} 
\mathcal{D}(G_{1}\dot{\vee}G_{2})=
\left[                 
  \begin{array}{cccc}     
    2(J-I) & 3J-2R_{1} & J &2J\\  
    3J-2R_{1}^{T} & 4(J-I)-2A(L(G_{1})) &  2J&3J\\  
    J & 2J & 2(J-I)& 3J-2R_{2}\\
    2J & 3J & 3J-2R_{2}^{T}& 4(J-I)-2A(L(G_{2}))
  \end{array}
\right],    \nonumber             
\end{equation}
where $J$ is the all-one matrix, and $I$ is the identity matrix of appropriate orders.

Let $\lambda_{i_{j}}\neq r_{i}$ be an eigenvalue of $A_{i}$ with an eigenvector $X_{i_{j}}$. Then, $X_{i_{j}}$ is orthogonal to the all-one matrix $J$, and $A_{i}X_{i_{j}}=\lambda_{i_{j}} X_{i_{j}}$.

Now, consider the vector
$\begin{matrix}
\phi_{1}=
\left[                 
  \begin{array}{cccc}   
    tX_{1_{j}}  \\  
    R_{1}^{T}X_{1_{j}} \\  
     0\\
     0
  \end{array}
\right]    \nonumber             
\end{matrix}$, which
is an eigenvector of $\mathcal{D}(G_{1}\dot{\vee}G_{2})$ with eigenvalues $\mu_{1}$, then from the equation $\mathcal{D}(G_{1}\dot{\vee}G_{2})\phi_{1}=\mu_{1}\phi_{1}$, we get
\begin{equation}\label{LGQQEqu5}
\begin{aligned} 
-2t-2(\lambda_{1_{j}}+r_{1})=\mu_{1} t,
\end{aligned}
\end{equation}
\begin{equation}\label{LGQQEqu6}
\begin{aligned} 
-2t-2(\lambda_{1_{j}}+r_{1})=\mu_{1}.
\end{aligned}
\end{equation}
Thus, solving $Eqs.(\ref{LGQQEqu5})$ and $Eqs.(\ref{LGQQEqu6})$ yields $\mu_{1}=-2(\lambda_{1_{j}}+r_{1}+1)$ and $\mu_{1}=0$, $j=2,3,\cdots,p_{1}$;

Then, consider the vector
$\begin{matrix}
\phi_{2}=
\left[                 
  \begin{array}{ccc}   
    0  \\  
    0 \\  
     tX_{2_{j}}\\
     R_{2}^{T}X_{2_{j}}
  \end{array}
\right]    \nonumber             
\end{matrix}$, which
is an eigenvector of $\mathcal{D}(G_{1}\dot{\vee}G_{2})$ with eigenvalues $\mu_{2}$, then from the equation $\mathcal{D}(G_{1}\dot{\vee}G_{2})\phi_{2}=\mu_{2}\phi_{2}$, we get
\begin{equation}\label{LGQQEqu10}
\begin{aligned} 
-2t-2(\lambda_{2_{j}}+r_{2})=\mu_{2} t,
\end{aligned}
\end{equation}
\begin{equation}\label{LGQQEqu11}
\begin{aligned} 
-2t-2(\lambda_{2_{j}}+r_{2})=\mu_{2}.
\end{aligned}
\end{equation}
Thus, solving $Eqs.(\ref{LGQQEqu10})$ and $Eqs.(\ref{LGQQEqu11})$ yields $\mu_{2}=-2(\lambda_{2_{j}}+r_{2}+1)$ and $\mu_{2}=0$, $j=2,3,\cdots,p_{2}$;

Let $Z_{i}$ be an eigenvector of $L(G_{i})$ with the eigenvalue $-2$. Then, by Lemma \ref{lem11}, $R_{i}Z_{i}=0$. Now, let
$\begin{matrix}
\psi_{1}=
\left[                 
  \begin{array}{cccc}   
    0  \\  
    Z_{1} \\  
     0\\
     0
  \end{array}
\right]    \nonumber             
\end{matrix}$ and
$\begin{matrix}
\psi_{2}=
\left[                 
  \begin{array}{cccc}   
    0  \\  
    0 \\  
     0\\
     Z_{2}
  \end{array}
\right]    \nonumber             
\end{matrix}$
is an eigenvector of $\mathcal{D}(G_{1}\dot{\vee} G_{2})$ corresponding to the eigenvalue $0$ with multiplicity $\frac{1}{2}p_{i}r_{i}-p_{i}$.
 \begin{equation*}
\begin{aligned} 
\mathcal{D}(G_{1}\dot{\vee} G_{2})\psi_{i}&=
\left[                 
  \begin{array}{cccc}     
     2(J-I) & 3J-2R_{1} & J &2J\\  
    3J-2R_{1}^{T} & 4(J-I)-2A(L(G_{1})) &  2J&3J\\  
    J & 2J & 2(J-I)& 3J-2R_{2}\\
    2J & 3J & 3J-2R_{2}^{T}& 4(J-I)-2A(L(G_{2}))
  \end{array}
\right]    \nonumber             
\psi_{i},\\  &
=0\psi_{i}.
\end{aligned}
\end{equation*}

Thus, forming eigenvectors of this type we get $p_{1}+p_{2}+\frac{1}{2}p_{1}r_{1}+\frac{1}{2}p_{2}r_{2}-4$ eigenvectors, and there remains 4. By
the construction, the remaining four are of the form
$\begin{matrix}
v=
\left[                 
  \begin{array}{cccc}   
    \alpha J  \\  
    \beta J \\  
     \gamma J\\
     \delta J
  \end{array}
\right]    \nonumber             
\end{matrix}$
for some $(\alpha,\beta,\gamma,\delta)\neq(0,0,0,0)$. If $\sigma$ be an eigenvalue of $\mathcal{D}(G_{1}\dot{\vee} G_{2})$ with an eigenvector $v$,
and then from $\mathcal{D}(G_{1}\dot{\vee} G_{2})v=\sigma v$ we can see that the remaining four are the eigenvalues of the matrix
\begin{equation} 
\left[                 
  \begin{array}{cccc}     
   2p_{1}-2 & \frac{3}{2}r_{1}p_{1}-2r_{1} & p_{2}& r_{2}p_{2} \\  
    3p_{1}-4 & 2r_{1}p_{1}-4r_{1} &  2p_{2}& \frac{3}{2}r_{2}p_{2}\\  
    p_{1} & r_{1}p_{1} & 2p_{2}-2 & \frac{3}{2}r_{2}p_{2}-2r_{2}\\
    2p_{1} & \frac{3}{2}r_{1}p_{1} & 3p_{2}-4 & 2r_{2}p_{2}-4r_{2}
  \end{array}
\right].    \nonumber             
\end{equation}
This completes the proof.\qed\end{proof}

\begin{thm}\label{thm21}
Let $G_{i}$ be an $r_{i}-$regular graph of order $p_{i}$ with the adjacency matrix $A_{i}$ and adjacency spectrum $\{r_{i},\lambda_{i_{2}},\lambda_{i_{3}},\cdots,\lambda_{i_{p_{i}}}\}$ for $i=1,2$. Then, the distance Laplacian spectrum of $G_{1}\dot{\vee}G_{2}$ is:\\
$(1)$ $\lambda_{1_{j}}+ \frac{5p_{1}}{2} + \frac{3p_{2}}{2} + \frac{7p_{1}r_{1}}{4}+ \frac{5p_{2}r_{2}}{4} - 2r_{1}-2 \pm \frac{\sqrt{4\lambda_{1_{j}}^2 + 2\lambda_{1_{j}}p_{1}r_{1} + 4\lambda_{1_{j}}p_{1} + 2\lambda_{1_{j}}p_{2}r_{2} + 4\lambda_{1_{j}}p_{2} + \frac{1}{4}p_{1}^2r_{1}^2 + p_{1}^2r_{1} + p_{1}^2 + }}{2}\\
\frac{ \overline{ \frac{1}{2}p_{1}p_{2}r_{1}r_{2}+ p_{1}p_{2}r_{1} + p_{1}p_{2}r_{2} + 2p_{1}p_{2} - 4p_{1}r_{1} - 8p_{1} + \frac{1}{4}p_{2}^2r_{2}^2 + p_{2}^2r_{2} + p_{2}^2 - 4p_{2}r_{2} - 8p_{2} + 16r_{1} + 16}}{2}$, $j=2,3,\cdots,p_{1}$;\\
$(2)$ $\lambda_{2_{j}}+ \frac{3p_{1}}{2} + \frac{5p_{2}}{2} +  \frac{5p_{1}r_{1}}{4}+ \frac{7p_{2}r_{2}}{4} - 2r_{2}-2 \pm \frac{\sqrt{4\lambda_{2_{j}}^2 + 2\lambda_{2_{j}}p_{1}r_{1} + 4\lambda_{2_{j}}p_{1} + 2\lambda_{2_{j}}p_{2}r_{2} + 4\lambda_{2_{j}}p_{2} + \frac{1}{4}p_{1}^2r_{1}^2 + p_{1}^2r_{1} + p_{1}^2 + }}{2}\\
\frac{ \overline{\frac{1}{2}p_{1}p_{2}r_{1}r_{2}+p_{1}p_{2}r_{1} + p_{1}p_{2}r_{2} + 2p_{1}p_{2} - 4p_{1}r_{1} - 8p_{1} + \frac{1}{4}p_{2}^2r_{2}^2 + p_{2}^2r_{2} + p_{2}^2 - 4p_{2}r_{2} - 8p_{2} + 16r_{2} + 16}}{2}$, $j=2,3,\cdots,p_{2}$;\\
$(3)$ $3p_{1}+2p_{1}r_{1}+2p_{2}+\frac{3}{2}p_{2}r_{2}-4r_{1}-4$ with multiplicity $\frac{1}{2}p_{1}r_{1}-p_{1}$;\\$(4)$ $2p_{1}+\frac{3}{2}p_{1}r_{1}+3p_{2}+2p_{2}r_{2}-4r_{2}-4$ with multiplicity $\frac{1}{2}p_{2}r_{2}-p_{2}$;\\
$(5)$ together with the four eigenvalues of
\begin{equation} 
\left[                 
  \begin{array}{ccc}     
  \begin{smallmatrix}
    \frac{3}{2}p_{1}r_{1}+p_{2}+p_{2}r_{2}-2r_{1} & 2r_{1}-\frac{3}{2}p_{1}r_{1} & -p_{2}& -p_{2}r_{2}\\  
    4-3p_{1} & 3p_{1}+2p_{2}+\frac{3}{2}p_{2}r_{2}-4 &  -2p_{2}& -\frac{3}{2}p_{2}r_{2}\\  
    -p_{1} & -p_{1}r_{1} & p_{1}+p_{1}r_{1}+\frac{3}{2}p_{2}r_{2}-2r_{2}&2r_{2}-\frac{3}{2}p_{2}r_{2}\\
    -2p_{1} & -\frac{3}{2}p_{1}r_{1} & 4-3p_{2}&2p_{1}+\frac{3}{2}p_{1}r_{1}+3p_{2}-4
  \end{smallmatrix}
  \end{array}
\right].    \nonumber             
\end{equation}
\end{thm}
\begin{proof} Let $R_{i}$ be the incidence matrix of $G_{i}$ and $A(L(G_{i}))$ be the adjacency matrix of $L(G_{i})$. Then, by a proper ordering of the vertices of $G_{1}\dot{\vee}G_{2}$,
its distance Laplacian matrix $\mathcal{L}(G_{1}\dot{\vee}G_{2})$ can be written in the form
$$\scriptsize{
\begin{aligned} 
\mathcal{L}(G_{1}\dot{\vee}G_{2})=
\left[                 
  \begin{array}{ccc}     
  \begin{smallmatrix}
    (2p_{1}+\frac{3}{2}p_{1}r_{1}+p_{2}+p_{2}r_{2}-2r_{1})I-2J & 2R_{1}-3J & -J &-2J \\  
    2R_{1}^{T}-3J & (3p_{1}+2p_{1}r_{1}+2p_{2}+\frac{3}{2}p_{2}r_{2}-4r_{1})I-4J+2A(L(G_{1})) &  -2J&  -3J\\  
    -J & -2J & B & 2R_{2}-3J\\
    -2J & -3J & 2R_{2}^{T}-3J & C\\
  \end{smallmatrix}
  \end{array}
\right],    \nonumber             
\end{aligned}}$$
where $B=(p_{1}+p_{1}r_{1}+2p_{2}+\frac{3}{2}p_{2}r_{2}-2r_{2})I-2J$, $C=(2p_{1}+\frac{3}{2}p_{1}r_{1}+3p_{2}+2p_{2}r_{2}-4r_{2})I-4J+2A(L(G_{2}))$, $J$ is the all-one matrix, and $I$ is the identity matrix of appropriate orders. We can prove the theorem in a similar way to Theorem
 \ref{thm20}.\qed\end{proof}

\begin{thm}\label{thm22}
Let $G_{i}$ be an $r_{i}-$regular graph of order $p_{i}$ with the adjacency matrix $A_{i}$ and adjacency spectrum $\{r_{i},\lambda_{i_{2}},\lambda_{i_{3}},\cdots,\lambda_{i_{p_{i}}}\}$ for $i=1,2$. Then, the distance signless Laplacian spectrum of $G_{1}\dot{\vee}G_{2}$ is:\\
$(1)$ $-\lambda_{1_{j}}+ \frac{5p_{1}}{2} + \frac{3p_{2}}{2} + \frac{7p_{1}r_{1}}{4}+ \frac{5p_{2}r_{2}}{4} - 4r_{1}-4 \pm \frac{\sqrt{4\lambda_{1_{j}}^2 - 2\lambda_{1_{j}}p_{1}r_{1} - 4\lambda_{1_{j}}p_{1} - 2\lambda_{1_{j}}p_{2}r_{2} - 4\lambda_{1_{j}}p_{2} + 16\lambda_{1_{j}}r_{1} + 16\lambda_{1_{j}} + }}{2}\\ \frac{ \overline{\frac{1}{4}p_{1}^2r_{1}^2 +p_{1}^2r_{1} + p_{1}^2 + \frac{1}{2}p_{1}p_{2}r_{1}r_{2} + p_{1}p_{2}r_{1} + p_{1}p_{2}r_{2} + 2p_{1}p_{2} - 4p_{1}r_{1}^2 - 8p_{1}r_{1} +\frac{1}{4}p_{2}^2r_{2}^2 + p_{2}^2r_{2} + p_{2}^2 - 4p_{2}r_{1}r_{2} - 8p_{2}r_{1} + 16r_{1}^2 + 16r_{1}}}{2}$, $j=2,3,\cdots,p_{1}$;\\
$(2)$ $-\lambda_{2_{j}}+ \frac{3p_{1}}{2} + \frac{5p_{2}}{2} + \frac{5p_{1}r_{1}}{4}+ \frac{7p_{2}r_{2}}{4} - 4r_{2}-4 \pm \frac{\sqrt{4\lambda_{2_{j}}^2 - 2\lambda_{2_{j}}p_{1}r_{1} - 4\lambda_{2_{j}}p_{1} - 2\lambda_{2_{j}}p_{2}r_{2} - 4\lambda_{2_{j}}p_{2} + 16\lambda_{2_{j}}r_{2} + 16\lambda_{2_{j}} + }}{2}\\ \frac{ \overline{\frac{1}{4}p_{1}^2r_{1}^2 +p_{1}^2r_{1} + p_{1}^2 + \frac{1}{2}p_{1}p_{2}r_{1}r_{2} + p_{1}p_{2}r_{1} + p_{1}p_{2}r_{2} + 2p_{1}p_{2} - 4p_{1}r_{1}r_{2} - 8p_{1}r_{2} +\frac{1}{4}p_{2}^2r_{2}^2 + p_{2}^2r_{2} + p_{2}^2 - 4p_{2}r_{2}^{2} - 8p_{2}r_{2} + 16r_{2}^2 + 16r_{2}}}{2}$, $j=2,3,\cdots,p_{2}$;\\
$(3)$ $3p_{1}+2p_{1}r_{1}+2p_{2}+\frac{3}{2}p_{2}r_{2}-4r_{1}-4$ with multiplicity $\frac{1}{2}p_{1}r_{1}-p_{1}$; \\
$(4)$ $2p_{1}+\frac{3}{2}p_{1}r_{1}+3p_{2}+2p_{2}r_{2}-4r_{2}-4$ with multiplicity $\frac{1}{2}p_{2}r_{2}-p_{2}$;\\
$(5)$ together with the four eigenvalues of
$$\scriptsize{
\begin{aligned} 
\left[                 
  \begin{array}{ccc}     
  \begin{smallmatrix}
    4p_{1}+\frac{3}{2}p_{1}r_{1}+p_{2}+p_{2}r_{2}-2r_{1}-4 & \frac{3}{2}p_{1}r_{1}-2r_{1} & p_{2}& p_{2}r_{2}\\  
    3p_{1}-4 & 3p_{1}+4p_{1}r_{1}+2p_{2}+\frac{3}{2}p_{2}r_{2}-8r_{1}-4 &  2p_{2}& \frac{3}{2}p_{2}r_{2}\\  
    p_{1} & p_{1}r_{1} & p_{1}+p_{1}r_{1}+4p_{2}+\frac{3}{2}p_{2}r_{2}-2r_{2}-4 & \frac{3}{2}p_{2}r_{2}-2r_{2}\\
    2p_{1} & \frac{3}{2}p_{1}r_{1} & 3p_{2}-4 & 2p_{1}+\frac{3}{2}p_{1}r_{1}+3p_{2}+4p_{2}r_{2}-8r_{2}-4
  \end{smallmatrix}
  \end{array}
\right].    \nonumber             
\end{aligned}
}$$
\end{thm}
\begin{proof} Let $R_{i}$ be the incidence matrix of $G_{i}$ and $A(L(G_{i}))$ be the adjacency matrix of $L(G_{i})$. Then, by a proper ordering of the vertices of $G_{1}\dot{\vee}G_{2}$,
its distance signless Laplacian matrix $\mathcal{Q}(G_{1}\dot{\vee}G_{2})$ can be written in the form
$$\scriptsize{
\begin{aligned} 
\mathcal{Q}(G_{1}\dot{\vee}G_{2})=
\left[                 
  \begin{array}{ccc}     
  \begin{smallmatrix}
    (2p_{1}+\frac{3}{2}p_{1}r_{1}+p_{2}+p_{2}r_{2}-2r_{1}-4)I+2J & 3J-2R_{1} & J &2J \\  
   3J-2R_{1}^{T} & (3p_{1}+2p_{1}r_{1}+2p_{2}+\frac{3}{2}p_{2}r_{2}-4r_{1}-8)I+4J-2A(L(G_{1})) &  2J&  3J\\  
    J & 2J & B & 3J-2R_{2}\\
    2J & 3J & 3J-2R_{2}^{T} & C\\
  \end{smallmatrix}
  \end{array}
\right],    \nonumber             
\end{aligned}}$$
where $B=(p_{1}+p_{1}r_{1}+2p_{2}+\frac{3}{2}p_{2}r_{2}-2r_{2}-4)I+2J$, $C=(2p_{1}+\frac{3}{2}p_{1}r_{1}+3p_{2}+2p_{2}r_{2}-4r_{2}-8)I+4J-2A(L(G_{2}))$,
 $J$ is the all-one matrix, and $I$ is the identity matrix of appropriate orders. We can prove the theorem in a similar way to Theorem
 \ref{thm20}.\qed\end{proof}
\subsection{The $\mathcal{D}-$spectrum, $\mathcal{L}-$spectrum and $\mathcal{Q}-$spectrum of $G_{1}\underline{\vee}G_{2}$}

\begin{thm}\label{thm23}
Let $G_{i}$ be an $r_{i}-$regular graph of order $p_{i}$ with the adjacency matrix $A_{i}$ and adjacency spectrum $\{r_{i},\lambda_{i_{2}},\lambda_{i_{3}},\cdots,\lambda_{i_{p_{i}}}\}$ for $i=1,2$. Then, the distance spectrum of $G_{1}\underline{\vee}G_{2}$ is:\\
$(1)$ $ -\lambda_{1_{j}}-3 \pm \sqrt{\lambda_{1_{j}}^{2}+6\lambda_{1_{j}}+4r_{1}+1}$, $j=2,3,\cdots,p_{1}$;\\
$(2)$ $-2\lambda_{2_{j}}-2r_{2}-2$, $j=2,3,\cdots,p_{2}$;\\
$(3)$ $-2$ with multiplicity $\frac{1}{2}p_{1}r_{1}-p_{1}$; \\
$(4)$ $0$ with multiplicity $\frac{1}{2}p_{2}r_{2}-1$; \\
$(5)$ together with the four eigenvalues of
\begin{equation} 
\left[                 
  \begin{array}{cccc}     
    4p_{1}-2r_{1}-4 & \frac{3}{2}p_{1}r_{1}-2r_{1} & 2p_{2}&\frac{3}{2}p_{2}r_{2} \\  
    3p_{1}-4 & p_{1}r_{1}-2 &  p_{2}&  p_{2}r_{2}\\  
    2p_{1} & \frac{1}{2}p_{1}r_{1} & 2p_{2}-2&\frac{3}{2}p_{2}r_{2}-2r_{2}\\
    3p_{1} & p_{1}r_{1} & 3p_{2}-4&2p_{2}r_{2}-4r_{2}
  \end{array}
\right].    \nonumber             
\end{equation}
\end{thm}

\begin{proof} Let $R_{i}$ be the incidence matrix of $G_{i}$ and $A(L(G_{i}))$ be the adjacency matrix of $L(G_{i})$, Then, by a proper ordering of the vertices of $G_{1}\underline{\vee}G_{2}$,
its distance matrix $\mathcal{D}(G_{1}\underline{\vee}G_{2})$ can be written in the form
\begin{equation} 
\mathcal{D}(G_{1}\underline{\vee}G_{2})=
\left[                 
  \begin{array}{cccc}     
   (4J-I)-2A_{1} & 3J-2R_{1}&2J &3J\\  
    3J-2R_{1}^{T} & 2(J-I)&J&2J\\  
    2J & J & 2(J-I) & 3J-2R_{2}\\
    3J & 2J & 3J-2R_{2}^{T} & 4(J-I)-2A(L(G_{2}))
  \end{array}
\right].    \nonumber             
\end{equation}
where $J$ is the all-one matrix, and $I$ is the identity matrix of appropriate orders. We can prove the theorem in a similar way to Theorem
 \ref{thm20}.\qed\end{proof}

\begin{thm}\label{thm24}
Let $G_{i}$ be an $r_{i}-$regular graph of order $p_{i}$ vertices with the adjacency matrix $A_{i}$ and adjacency spectrum $\{r_{i},\lambda_{i_{2}},\lambda_{i_{3}},\cdots,\lambda_{i_{p_{i}}}\}$ for $i=1,2$. Then, the distance Laplacian spectrum of $G_{1}\underline{\vee}G_{2}$ is:\\
$(1)$ $\lambda_{1_{j}}+\frac{7p_{1}}{2}+ \frac{3p_{2}}{2} + \frac{5p_{1}r_{1}}{4}+ \frac{5p_{2}r_{2}}{4}- 2r_{1}-2 \pm\frac{\sqrt{4\lambda_{1_{j}}^2 + 2\lambda_{1_{j}}p_{1}r_{1} + 4\lambda_{1_{j}}p_{1} + 2\lambda_{1_{j}}p_{2}r_{2} + 4\lambda_{1_{j}}p_{2} - 16\lambda_{1_{j}}r_{1} + 32\lambda_{1_{j}} + \frac{1}{4}p_{1}^2r_{1}^2 }}{2}\\
\frac{\overline{+ p_{1}^2r_{1} + p_{1}^2 + \frac{1}{2}p_{1}p_{2}r_{1}r_{2} + p_{1}p_{2}r_{1} + p_{1}p_{2}r_{2} + 2p_{1}p_{2} - 4p_{1}r_{1}^2 - 4p_{1}r_{1} + 8p_{1} + \frac{1}{4}p_{2}^2r_{2}^2 + p_{2}^2r_{2} + p_{2}^2 - 4p_{2}r_{1}r_{2} - 8p_{2}r_{1} + 4p_{2}r_{2} + 8p_{2} + 16r_{1}^2}}{2}\\
\frac{\overline{ - 16r_{1} + 16}}{2}$, $j=2,3,\cdots,p_{1}$;\\
$(2)$ $\lambda_{2_{j}}+\frac{5p_{1}}{2}+ \frac{5p_{2}}{2} + \frac{3p_{1}r_{1}}{4}+ \frac{7p_{2}r_{2}}{4}- 2r_{2}-2 \pm\frac{\sqrt{4\lambda_{2_{j}}^2 + 2\lambda_{2_{j}}p_{1}r_{1} + 4\lambda_{2_{j}}p_{1} + 2\lambda_{2_{j}}p_{2}r_{2} + 4\lambda_{2_{j}}p_{2}+\frac{1}{4}p_{1}^2r_{1}^2 + p_{1}^2r_{1} + p_{1}^2 + }}{2}\\
\frac{\overline{\frac{1}{2}p_{1}p_{2}r_{1}r_{2}+ p_{1}p_{2}r_{1} + p_{1}p_{2}r_{2} + 2p_{1}p_{2} - 4p_{1}r_{1} - 8p_{1} + \frac{1}{4}p_{2}^2r_{2}^2 + p_{2}^2r_{2} + p_{2}^2 - 4p_{2}r_{2} - 8p_{2} + 16r_{2} + 16}}{2}$, $j=2,3,\cdots,p_{2}$;\\
$(3)$ $3p_{1}+p_{1}r_{1}+p_{2}+p_{2}r_{2}-4$ with multiplicity $\frac{1}{2}p_{1}r_{1}-p_{1}$;\\
$(4)$ $3p_{1}+p_{1}r_{1}+2p_{2}+2p_{2}r_{2}-4r_{2}-4$ with multiplicity $\frac{1}{2}p_{2}r_{2}-p_{2}$;\\
$(5)$ together with the four eigenvalues of
\begin{equation} 
\left[                 
  \begin{array}{cccc}     
  \begin{smallmatrix}
    \frac{3}{2}p_{1}r_{1}+2p_{2}+\frac{3}{2}p_{2}r_{2}-2r_{1} & 2r_{1}-\frac{3}{2}p_{1}r_{1} & -2p_{2} & -\frac{3}{2}p_{2}r_{2}\\  
   4-3p_{1} & 3p_{1}+p_{2}+p_{2}r_{2}-4 &  -p_{2} & -p_{2}r_{2}\\  
    -2p_{1} & -\frac{1}{2}p_{1}r_{1} & 2p_{1}+\frac{1}{2}p_{1}r_{1}+\frac{3}{2}p_{2}r_{2}-2r_{2} & 2r_{2}-\frac{3}{2}p_{2}r_{2}\\
    -3p_{1} & -p_{1}r_{1} & 4-3p_{2} & 3p_{1}+p_{1}r_{1}+3p_{2}-4
  \end{smallmatrix}
  \end{array}
\right].    \nonumber             
\end{equation}
\end{thm}
\begin{proof} Let $R_{i}$ be the incidence matrix of $G_{i}$ and $A(L(G_{i}))$ be the adjacency matrix of $L(G_{i})$, Then, by a proper ordering of the vertices of $G_{1}\underline{\vee}G_{2}$,
its distance Laplacian matrix $\mathcal{L}(G_{1}\underline{\vee}G_{2})$ can be written in the form
$$\scriptsize{
\begin{aligned} 
\mathcal{L}(G_{1}\underline{\vee}G_{2})=
\left[                 
  \begin{array}{ccc}     
  \begin{smallmatrix}
    (4p_{1}+\frac{3}{2}p_{1}r_{1}+2p_{2}+\frac{3}{2}p_{2}r_{2}-4r_{1})I-4J+2A_{1} & 2R_{1}-3J & -2J &-3J\\  
    2R_{1}^{T}-3J & (3p_{1}+p_{1}r_{1}+p_{2}+p_{2}r_{2}-4)I-2J &  -J&-2J\\  
    -2J & -J & B & 2R_{2}-3J\\
    -3J & -2J & 2R_{2}^{T}-3J &C
  \end{smallmatrix}
  \end{array}
\right],    \nonumber             
\end{aligned}}$$
where $B=(2p_{1}+\frac{1}{2}p_{1}r_{1}+2p_{2}+\frac{3}{2}p_{2}r_{2}-2r_{2})I-2J$, $C=(3p_{1}+p_{1}r_{1}+3p_{2}+2p_{2}r_{2}-4r_{2})I-4J+2A(L(G_{2}))$, $J$ is the all-one matrix, and $I$ is the identity matrix of appropriate orders. We can prove the theorem in a similar way to Theorem
 \ref{thm20}.\qed\end{proof}

\begin{thm}\label{thm25}
Let $G_{i}$ be an $r_{i}-$regular graph of order $p_{i}$ with the adjacency matrix $A_{i}$ and adjacency spectrum $\{r_{i},\lambda_{i_{2}},\lambda_{i_{3}},\cdots,\lambda_{i_{p_{i}}}\}$ for $i=1,2$. Then, the distance signless Laplacian spectrum of $G_{1}\underline{\vee}G_{2}$ is:\\
$(1)$ $-\lambda_{1_{j}}+\frac{7p_{1}}{2}+ \frac{3p_{2}}{2} + \frac{5p_{1}r_{1}}{4}+ \frac{5p_{2}r_{2}}{4}- 2r_{1}-8 \pm\frac{\sqrt{4\lambda_{1_{j}}^2 - 2\lambda_{1_{j}}p_{1}r_{1} - 4\lambda_{1_{j}}p_{1} - 2\lambda_{1_{j}}p_{2}r_{2} - 4\lambda_{1_{j}}p_{2} + 16\lambda_{1_{j}}r_{1} + 16\lambda_{1_{j}} + }}{2}\\
\frac{\overline{\frac{1}{4}p_{1}^2r_{1}^2 +  p_{1}^2r_{1} + p_{1}^2 + \frac{1}{2}p_{1}p_{2}r_{1}r_{2} + p_{1}p_{2}r_{1} + p_{1}p_{2}r_{2} + 2p_{1}p_{2} - 4p_{1}r_{1}^2 - 8p_{1}r_{1} + \frac{1}{4}p_{2}^2r_{2}^2 + p_{2}^2r_{2} + p_{2}^2 - 4p_{2}r_{1}r_{2} - 8p_{2}r_{1} + 16r_{1}^2 + 16r_{1}}}{2}$, $j=2,3,\cdots,p_{1}$;\\
$(2)$ $-\lambda_{2_{j}}+\frac{5p_{1}}{2}+ \frac{5p_{2}}{2} + \frac{3p_{1}r_{1}}{4}+ \frac{7p_{2}r_{2}}{4}- 4r_{2}-4 \pm\frac{\sqrt{4\lambda_{2_{j}}^2 - 2\lambda_{2_{j}}p_{1}r_{1} - 4\lambda_{2_{j}}p_{1} - 2\lambda_{2_{j}}p_{2}r_{2} - 4\lambda_{2_{j}}p_{2} + 16\lambda_{2_{j}}r_{2} + 16\lambda_{2_{j}} + }}{2}\\
\frac{\overline{\frac{1}{4}p_{1}^2r_{1}^2 +  p_{1}^2r_{1} + p_{1}^2 + \frac{1}{2}p_{1}p_{2}r_{1}r_{2} + p_{1}p_{2}r_{1} + p_{1}p_{2}r_{2} + 2p_{1}p_{2} - 4p_{1}r_{1}r_{2} - 8p_{1}r_{2} + \frac{1}{4}p_{2}^2r_{2}^2 + p_{2}^2r_{2} + p_{2}^2 - 4p_{2}r_{2}^{2} - 8p_{2}r_{2} + 16r_{2}^2 + 16r_{2}}}{2}$, $j=2,3,\cdots,p_{2}$;\\
$(3)$ $3p_{1}+p_{1}r_{1}+p_{2}+p_{2}r_{2}-8$ with multiplicity $\frac{1}{2}p_{1}r_{1}-p_{1}$;\\
$(4)$ $3p_{1}+p_{1}r_{1}+3p_{2}+2p_{2}r_{2}-4r_{2}-4$ with multiplicity $\frac{1}{2}p_{2}r_{2}-p_{2}$;\\
$(5)$ together with the four eigenvalues of
$$\scriptsize{
\begin{aligned} 
\left[                 
  \begin{array}{cccc}     
  \begin{smallmatrix}
    8p_{1}+\frac{3}{2}p_{1}r_{1}+2p_{2}+\frac{3}{2}p_{2}r_{2}-6r_{1}-8 & \frac{3}{2}p_{1}r_{1}-2r_{1} & 2p_{2} & \frac{3}{2}p_{2}r_{2}\\  
   3p_{1}-4 & 3p_{1}+2p_{1}r_{1}+p_{2}+p_{2}r_{2}-8 &  p_{2} & p_{2}r_{2}\\  
    2p_{1} & \frac{1}{2}p_{1}r_{1} & 2p_{1}+\frac{1}{2}p_{1}r_{1}+4p_{2}+\frac{3}{2}p_{2}r_{2}-2r_{2}-4 & \frac{3}{2}p_{2}r_{2}-2r_{2}\\
    3p_{1} & p_{1}r_{1} & 3p_{2}-4 & 3p_{1}+p_{1}r_{1}+3p_{2}+4p_{2}r_{2}-8r_{2}-4
  \end{smallmatrix}
  \end{array}
\right].    \nonumber             
\end{aligned}}$$
\end{thm}
\begin{proof} Let $R_{i}$ be the incidence matrix of $G_{i}$ and $A(L(G_{i}))$ be the adjacency matrix of $L(G_{i})$, Then, by a proper ordering of the vertices of $G_{1}\underline{\vee}G_{2}$,
its distance signless Laplacian matrix $\mathcal{Q}(G_{1}\underline{\vee}G_{2})$ can be written in the form
$$\scriptsize{
\begin{aligned} 
\mathcal{Q}(G_{1}\underline{\vee}G_{2})=
\left[                 
  \begin{array}{ccc}     
  \begin{smallmatrix}
    (4p_{1}+\frac{3}{2}p_{1}r_{1}+2p_{2}+\frac{3}{2}p_{2}r_{2}-4r_{1}-8)I+4J-2A_{1} & 3J-2R_{1} & 2J &3J\\  
    3J-2R_{1}^{T} & (3p_{1}+p_{1}r_{1}+p_{2}+p_{2}r_{2}-8)I+2J &  J & 2J\\  
    2J & J & B & 3J-2R_{2}\\
    3J & 2J & 3J-2R_{2}^{T} & C
  \end{smallmatrix}
  \end{array}
\right],    \nonumber             
\end{aligned}}$$
where $B=(2p_{1}+\frac{1}{2}p_{1}r_{1}+2p_{2}+\frac{3}{2}p_{2}r_{2}-2r_{2}-4)I+2J$, $C=(3p_{1}+p_{1}r_{1}+3p_{2}+2p_{2}r_{2}-4r_{2}-8)I+4J-2A(L(G_{2}))$,
$J$ is the all-one matrix, and $I$ is the identity matrix of appropriate orders. We can prove the theorem in a similar way to Theorem
\ref{thm20}.\qed\end{proof}

\section{Bounds on generalized distance spectral radius}
 In this section, we establish some bounds on generalized distance spectral radius of a simple, connected graph $G$. To the begin, we discuss about the lower and upper bounds of $\rho(G)$ involving transmission degree and second transmission degree.

 \begin{thm}\label{thm41}
 Let $G$ be a simple, connected graph. Then
 \begin{equation*}
\begin{aligned}  
\rho(G)\leq\frac{\alpha Tr_{\max}+\sqrt{\alpha^{2}Tr_{\max}^{2}+4(1-\alpha)T_{\max}}}{2},
\end{aligned}
\end{equation*}
where $Tr_{\max}$ and $T_{\max}$ are the maximum transmission degree and the maximum second
transmission degree of $G$, respectively. Moreover, the equality holds if and only if $G$
is transmission regular graph.
\end{thm}
 \begin{proof}
 Since $\mathcal{D}_{\alpha}(G)=\alpha Tr(G)+(1-\alpha)\mathcal{D}(G),~0\leq\alpha\leq1$,  by a simple calculation, we have $r_{v_{i}}(\mathcal{D}_{\alpha}(G))=Tr_{i}$, $r_{v_{i}}(Tr^{2})=r_{v_{i}}(Tr\mathcal{D})=Tr^{2}(v_{i})$ and $r_{v_{i}}(\mathcal{D}^{2})=r_{v_{i}}(\mathcal{D}Tr)=T_{v_{i}}$. Then
   \begin{equation*}
\begin{aligned}  
r_{v_{i}}(\mathcal{D}_{\alpha}^{2}(G))&=r_{v_{i}}(\alpha^{2}Tr^{2}+\alpha(1-\alpha)Tr\mathcal{D}+\alpha(1-\alpha)\mathcal{D}Tr+(1-\alpha)^{2}\mathcal{D}^{2})\\&
= r_{v_{i}}(\alpha Tr(\alpha Tr+(1-\alpha)\mathcal{D})+\alpha(1-\alpha)r_{v_{i}}(\mathcal{D}Tr)+(1-\alpha)^{2}r_{v_{i}}(\mathcal{D}^{2})\\&
=\alpha Tr_{i}r_{v_{i}}(\mathcal{D}_{\alpha}(G))+(1-\alpha)T_{i}\\&
\leq \alpha Tr_{\max}r_{v_{i}}(\mathcal{D}_{\alpha}(G))+(1-\alpha)T_{\max}.
\end{aligned}
\end{equation*}
So we have
\begin{equation*}
\begin{aligned}  
r_{v_{i}}[\mathcal{D}_{\alpha}^{2}(G)- \alpha Tr_{\max}r_{v_{i}}(\mathcal{D}_{\alpha}(G))]
\leq(1-\alpha)T_{\max}.
\end{aligned}
\end{equation*}
By Lemma \ref{lem12}, we have
\begin{equation*}
\begin{aligned}  
\rho^{2}(G)- \alpha Tr_{\max}\rho(G)-(1-\alpha)T_{\max}\leq0.
\end{aligned}
\end{equation*}
and then by Lemma \ref{lem12}, the result follows. In order to get the equality, all inequalities
in the above should be equalities. That is $Tr_{i}=Tr_{\max}$ and $T_{i}=T_{\max}$ holds for any
vertex $v_{i}$. So by Lemma \ref{lem12}, $G$ is transmission regular.

Conversely, when $G$ is transmission regular, it is easy to check that the equality
holds.
 \qed\end{proof}
 \begin{thm}\label{thm42}
 Let $G$ be a simple, connected graph. Then
 \begin{equation*}
\begin{aligned}  
\rho(G)\geq\frac{\alpha Tr_{\min}+\sqrt{\alpha^{2}Tr_{\min}^{2}+4(1-\alpha)T_{\min}}}{2},
\end{aligned}
\end{equation*}
where $Tr_{\min}$ and $T_{\min}$ are the minimum transmission degree and the minimum second
transmission degree of $G$, respectively. Moreover, the equality holds if and only if $G$
is transmission regular graph.
\end{thm}
\begin{proof}
Similar to the proof of Theorem \ref{thm41}.
\qed\end{proof}

\begin{thm}\label{thm43}
If the transmission degree sequence of $G$ is $\{Tr_{1},Tr_{2},\cdots,Tr_{n}\}$, then
\begin{equation}\label{gong1}
\begin{aligned}  
\rho(G)\geq\sqrt{\frac{\sum_{i=1}^{n}Tr_{i}^{2}}{n}}.
\end{aligned}
\end{equation}
with equality holding if and only if $G$ is transmission regular graph.
\end{thm}
\begin{proof}  Let $X=(x_{1},x_{2},\cdots,x_{n} )^{T}$ be the generalized distance Perron vector of $G$ and $C =\frac{1}{\sqrt{n}} (1,1,\cdots,1)^{ T}$. Then
\begin{equation*}
\begin{aligned}  
\rho(G)=\sqrt{\rho(G)^{2}}=\sqrt{X^{T}\mathcal{D}^{2}_{\alpha}(G)X}\geq\sqrt{C^{T}\mathcal{D}^{2}_{\alpha}(G)C},
\end{aligned}
\end{equation*}\\
and
\begin{equation*}
\begin{aligned}  
C^{T}\mathcal{D}^{2}_{\alpha}(G)C&=C^{T}(\alpha Tr+(1-\alpha)\mathcal{D})^{2}C
\\&=\alpha^{2}C^{T}Tr^{2}C+\alpha(1-\alpha)C^{T}\mathcal{D}TrC+\alpha(1-\alpha)C^{T}Tr\mathcal{D}C+(1-\alpha)^{2}C^{T}\mathcal{D}^{2}C.
\end{aligned}
\end{equation*}
We now have
\begin{equation*}
\begin{aligned}  
C^{T}Tr=C^{T}\mathcal{D}=\frac{1}{\sqrt{n}}(Tr_{1},Tr_{2},\cdots,Tr_{n}),\\
TrC=\mathcal{D}C=\frac{1}{\sqrt{n}}(Tr_{1},Tr_{2},\cdots,Tr_{n})^{T}.
\end{aligned}
\end{equation*}
Hence
\begin{equation*}
\begin{aligned}  
C^{T}Tr^{2}C=C^{T}\mathcal{D}TrC=C^{T}Tr\mathcal{D}C=C^{T}\mathcal{D}^{2}C=\frac{\sum_{i=1}^{n}Tr_{i}^{2}}{n}.
\end{aligned}
\end{equation*}
Thus,
\begin{equation*}
\begin{aligned}  
\rho(G)\geq\sqrt{\alpha^{2}\frac{\sum_{i=1}^{n}Tr_{i}^{2}}{n}+2\alpha(1-\alpha)\frac{\sum_{i=1}^{n}Tr_{i}^{2}}{n}+(1-\alpha)^{2}\frac{\sum_{i=1}^{n}Tr_{i}^{2}}{n}}
=\sqrt{\frac{\sum_{i=1}^{n}Tr_{i}^{2}}{n}}.
\end{aligned}
\end{equation*}
Now assume that G is $k$-transmission regular graph. Then by the Theorem of Frobenius
\cite{12}, $k$ is the simple and the greatest eigenvalue of $\mathcal{D}_{\alpha}(G)$. But then $\rho(G)=\sqrt{\frac{\sum_{i=1}^{n}Tr_{i}^{2}}{n}}=\sqrt{\frac{nk^{2}}{n}}=k$, and hence equality in (\ref{gong1}) holds,  Conversely, if equality in (\ref{gong1}) holds, then
$C$ is the eigenvector corresponding to $\rho(G)$. Thus, $G$ is transmission regular.
\qed\end{proof}

\begin{thm}\label{thm44}
 If the transmission degree sequence and the second transmission degree sequence of $G$ are $\{Tr_{1},Tr_{2},\cdots,Tr_{n}\}$ and $\{T_{1},T_{2},\cdots,T_{n}\}$, respectively, then
 \begin{equation}\label{gong2}
\begin{aligned}  
\rho(G)\leq\max_{1\leq i,j\leq n}\left\{\frac{\alpha (Tr_{i}+Tr_{j})+\sqrt{\alpha^{2}(Tr_{i}-Tr_{j})^{2}+4(1-\alpha)^{2}(\frac{T_{i}}{Tr_{i}})(\frac{T_{j}}{Tr_{j}})}}{2}\right\},
\end{aligned}
\end{equation}
with equality holding if and only if $G$ is a transmission regular graph.
\end{thm}
\begin{proof}
 Let $X=(x_{1},x_{2},\cdots,x_{n})^{T}$ be an eigenvector of $Tr^{-1}\mathcal{D}_{\alpha}(G)Tr$ corresponding to $\rho(G)$. Suppose $x_{s}=\max\{x_{i}|i=1,2,\cdots,n\}$ and $x_{t}=\max\{x_{i}|x _{i}\neq x_{s},i=1,2,\cdots,n\}$. Now the $(i,j)$th entry of $Tr^{-1}\mathcal{D}_{\alpha}(G)Tr$ is
 \begin{equation*}
   \begin{cases}
   \alpha Tr_{i},&\mbox{if $i=j$}.\\
   (1-\alpha)\frac{Tr_{j}}{Tr_{i}}d_{ij},&\mbox{if $i\neq j$}.\\
   \end{cases}
  \end{equation*}
 Since,
 \begin{equation}\label{gong3}
\begin{aligned}  
Tr^{-1}\mathcal{D}_{\alpha}(G)TrX=\rho(G)X,
\end{aligned}
\end{equation}
from the $s$th equation of (\ref{gong3}), we have
\begin{equation}\label{gong4}
\begin{aligned}  
(\rho(G)-\alpha Tr_{s})x_{s}=(1-\alpha)\sum_{i=1}^{n}\frac{Tr_{i}}{Tr_{s}}d_{si}x_{i}\leq(1-\alpha)\frac{x_{t}}{Tr_{s}}\sum_{i=1}^{n}d_{si}Tr_{i}=(1-\alpha)\frac{T_{s}}{Tr_{s}}x_{t}.
\end{aligned}
\end{equation}
Similarly, from the $t$th equation of (\ref{gong3}), we have
\begin{equation}\label{gong5}
\begin{aligned}  
(\rho(G)-\alpha Tr_{t})x_{t}=(1-\alpha)\sum_{i=1}^{n}\frac{Tr_{i}}{Tr_{t}}d_{ti}x_{i}\leq(1-\alpha)\frac{x_{s}}{Tr_{t}}\sum_{i=1}^{n}d_{ti}Tr_{i}=(1-\alpha)\frac{T_{t}}{Tr_{t}}x_{s}.
\end{aligned}
\end{equation}
Combining (\ref{gong4}) and (\ref{gong5}) we get $(\rho(G)-\alpha Tr_{s})(\rho(G)-\alpha Tr_{t})x_{s}x_{t}\leq(1-\alpha)^{2}\frac{T_{s}}{Tr_{s}}\frac{T_{t}}{Tr_{t}}x_{t}x_{s}$, which implies that
\begin{equation*}
\begin{aligned}  
\rho^{2}(G)-\alpha(Tr_{s}+Tr_{t})\rho(G)+\alpha^{2}Tr_{s}Tr_{t}-(1-\alpha)^{2}\left(\frac{T_{s}}{Tr_{s}}\right)\left(\frac{T_{t}}{Tr_{t}}\right)\leq0.
\end{aligned}
\end{equation*}
i.e.,
\begin{equation*}
\begin{aligned}  
\rho(G)\leq\frac{\alpha (Tr_{i}+Tr_{j})+\sqrt{\alpha^{2}(Tr_{i}-Tr_{j})^{2}+4(1-\alpha)^{2}(\frac{T_{i}}{Tr_{i}})(\frac{T_{j}}{Tr_{j}})}}{2}.
\end{aligned}
\end{equation*}
Hence,
\begin{equation*}
\begin{aligned}  
\rho(G)\leq\max_{1\leq i,j\leq n}\left\{\frac{\alpha (Tr_{i}+Tr_{j})+\sqrt{\alpha^{2}(Tr_{i}-Tr_{j})^{2}+4(1-\alpha)^{2}(\frac{T_{i}}{Tr_{i}})(\frac{T_{j}}{Tr_{j}})}}{2}\right\}.
\end{aligned}
\end{equation*}
Now we assume that $G$ is $k$-transmission regular graph. Then $Tr_{i}=k, \forall i=1,2,\cdots,n$, and $\rho(G)=k$. Hence, equality in (\ref{gong2}) holds.

Conversely, if $\rho(G)$ attains the upper bound of (\ref{gong2}), then all equalities in the above argument must hold. In particular, from (\ref{gong4}) and (\ref{gong5}), $x_{1} = x _{2} = \cdots = x _{n}$. Hence, $\rho(G)=\alpha Tr_{1}+(1-\alpha)\frac{T_{1}}{Tr_{1}}=\alpha Tr_{2}+(1-\alpha)\frac{T_{2}}{Tr_{2}}= \cdots = \alpha Tr_{n}+(1-\alpha)\frac{T_{n}}{Tr_{n}}$. It means $G$ is transmission regular graph.
\qed\end{proof}
\begin{thm}\label{C6}
 If the transmission degree sequence and the second transmission degree sequence of $G$ are $\{Tr_{1},Tr_{2},\cdots,Tr_{n}\}$ and $\{T_{1},T_{2},\cdots,T_{n}\}$, respectively, then
 \begin{equation*}
\begin{aligned}  
\rho(G)\geq\min_{1\leq i,j\leq n}\left\{\frac{\alpha (Tr_{i}+Tr_{j})+\sqrt{\alpha^{2}(Tr_{i}-Tr_{j})^{2}+4(1-\alpha)^{2}(\frac{T_{i}}{Tr_{i}})(\frac{T_{j}}{Tr_{j}})}}{2}\right\},
\end{aligned}
\end{equation*}
with equality holding if and only if $G$ is a transmission regular graph.
\end{thm}
\begin{proof}
Similar to the proof of Theorem \ref{thm44}.
\qed\end{proof}
In the following, we discuss about the lower and upper bounds of $\rho(G)$ involving maximum and minimum vertex degrees.
\begin{thm}\label{thm46}
 If $G$ is graph of order $n$, having maximum degree $\bigtriangleup_{1}$ and second maximum degree $\bigtriangleup_{2}$, then
\begin{small}
 \begin{equation}\label{gong6}
\begin{aligned}  
\rho(G)\geq\frac{\alpha(4n-4-\bigtriangleup_{1}-\bigtriangleup_{2})+\sqrt{\alpha^{2}(4n-4-\bigtriangleup_{1}-\bigtriangleup_{2})^{2}-4(2\alpha-1)(2n-2-\bigtriangleup_{1})(2n-2-\bigtriangleup_{2})}}{2},
\end{aligned}
\end{equation}
\end{small}
with equality holding if and only if $G$ is a regular graph with diameter less than or equal to $2$.
\end{thm}
\begin{proof}
 Let $X=(x_{1},x_{2},\cdots,x_{n})^{T}$ be the generalized distance Perron vector of $G$ such that $x_{i}=\min\{x_{k}|k=1,2,\cdots,n\}$ and $x_{j}=\min\{x_{k}|x _{k}\neq x_{i},k=1,2,\cdots,n\}$. From the eigenequation for the component $x_{i}$, we have
\begin{equation*}
\begin{aligned}  
\rho(G)x_{i}&=\sum_{k=1}^{n}d_{ik}((1-\alpha)x_{k}+\alpha x_{i})\\& \geq d_{vi}((1-\alpha)x_{j}+\alpha x_{i})+2(n-1-d_{vi})((1-\alpha)x_{j}+\alpha x_{i})\\&
=(2n-2-d_{vi})((1-\alpha)x_{j}+\alpha x_{i}).
\end{aligned}
\end{equation*}
i.e.,
\begin{equation}\label{gong7}
\begin{aligned}  
(\rho(G)-\alpha(2n-2-d_{vi}))x_{i}\geq(1-\alpha)(2n-2-d_{vi})x_{j}.
\end{aligned}
\end{equation}
Analogously for the component $x_{j}$ , we have
 \begin{equation*}
\begin{aligned}  
\rho(G)x_{j}&=\sum_{k=1}^{n}d_{jk}((1-\alpha)x_{k}+\alpha x_{j})\\& \geq d_{vj}((1-\alpha)x_{i}+\alpha x_{j})+2(n-1-d_{vj})((1-\alpha)x_{i}+\alpha x_{j})\\&
=(2n-2-d_{vj})((1-\alpha)x_{i}+\alpha x_{j}).
\end{aligned}
\end{equation*}
 i.e.,
\begin{equation}\label{gong8}
\begin{aligned}  
(\rho(G)-\alpha(2n-2-d_{vj}))x_{j}\geq(1-\alpha)(2n-2-d_{vj})x_{i}.
\end{aligned}
\end{equation}
 Combining (\ref{gong7}) and (\ref{gong8}) we get $(\rho(G)-\alpha(2n-2-d_{vi}))(\rho(G)-\alpha(2n-2-d_{vj}))x_{i}x_{j}
 \geq(1-\alpha)^{2}(2n-2-d_{vi})(2n-2-d_{vj})x_{i}x_{j}$, which implies that
 \begin{equation*}
\begin{aligned}  
\rho^{2}(G)-\alpha(4n-4-d_{vi}-d_{vj})\rho(G)+(2\alpha-1)(2n-2-d_{vi})(2n-2-d_{vj})\geq0.
\end{aligned}
\end{equation*}
  i.e.,
  \begin{small}
 \begin{equation*}
\begin{aligned}  
\rho(G)&\geq\frac{\alpha(4n-4-d_{vi}-d_{vj})+\sqrt{\alpha^{2}(4n-4-d_{vi}
-d_{vj})^{2}-4(2\alpha-1)(2n-2-d_{vi})(2n-2-d_{vj})}}{2}\\&
=\frac{\alpha(4n-4-\bigtriangleup_{1}-\bigtriangleup_{2})+\sqrt{\alpha^{2}(4n-4-\bigtriangleup_{1}-
\bigtriangleup_{2})^{2}-4(2\alpha-1)(2n-2-\bigtriangleup_{1})(2n-2-\bigtriangleup_{2})}}{2},
\end{aligned}
\end{equation*}
\end{small}
 Suppose that $\rho(G)=\frac{\alpha(4n-4-\bigtriangleup_{1}-\bigtriangleup_{2})+\sqrt{\alpha^{2}(4n-4-\bigtriangleup_{1}-
\bigtriangleup_{2})^{2}-4(2\alpha-1)(2n-2-\bigtriangleup_{1})(2n-2-\bigtriangleup_{2})}}{2}$. Then equality must hold in each of the
inequalities in the above argument. This will imply that $G$ is a regular graph with
diameter less than or equal to $2$. Conversely, let $G$ be a regular graph and diameter
of G be at most $2$. Thus, all the components $x_{i}$ are equal. If $d(G) = 1$, $G\cong K_{n}$,
and $\rho(G)=n-1$. Thus, equality in (\ref{gong6}) holds. If $d(G)=2$,
we get $\rho(G)x_{i}=(2n-2-d_{vi})x_{i}$. Thus, $\rho(G)=2n-2-d_{vi}$, and the equality in (\ref{gong6}) holds.
\qed\end{proof}
\begin{thm}\label{thm47}
 If $G$ is graph of order $n$, having minimum degree $\delta_{1}$ and second minimum degree $\delta_{2}$,  If $d$ is the diameter of $G$, then
\begin{small}
 \begin{equation}\label{gong9}
\begin{aligned}  
\rho(G)\leq&\frac{\alpha[2nd-(d-1)(d+\delta_{1}+\delta_{2})-2]+\sqrt{\alpha^{2}[2nd-(d-1)(d+\delta_{1}+\delta_{2})-2]^{2}
-4(2\alpha-1)[nd-(d-1)}}{2}\\&
\frac{\overline{(\frac{d}{2}+\delta_{1})-1][nd-(d-1)(\frac{d}{2}+\delta_{2})-1]}}{2}.
\end{aligned}
\end{equation}
\end{small}
with equality holding if and only if $G$ is a regular graph with diameter less than or equal to $2$.
\end{thm}
\begin{proof}
Let $X=(x_{1},x_{2},\cdots,x_{n})^{T}$ be the generalized distance Perron vector of $G$ such that $x_{i}=\max\{x_{k}|k=1,2,\cdots,n\}$ and $x_{j}=\max\{x_{k}|x _{k}\neq x_{i},k=1,2,\cdots,n\}$. From the eigenequation for the component $x_{i}$, we have
\begin{equation*}
\begin{aligned}  
\rho(G)x_{i}&=\sum_{k=1}^{n}d_{ik}[(1-\alpha)x_{k}+\alpha x_{i}]\\& \leq d_{vi}[(1-\alpha)x_{j}+\alpha x_{i}]+2[(1-\alpha)x_{j}+\alpha x_{i}]+3[(1-\alpha)x_{j}+\alpha x_{i}]+\cdots+(d-1)[(1-\alpha)\\&x_{j}+\alpha x_{i}]+d[n-1-d_{vi}-(d-2)][(1-\alpha)x_{j}+\alpha x_{i}]\\&
=\left[nd-\frac{d(d-1)}{2}-1-d_{vi}(d-1)\right][(1-\alpha)x_{j}+\alpha x_{i}].
\end{aligned}
\end{equation*}
i.e.,
\begin{equation}\label{gong10}
\begin{aligned}  
\left\{\rho(G)-\alpha\left[nd-\frac{d(d-1)}{2}-1-d_{vi}(d-1)\right]\right\}x_{i}\leq(1-\alpha)\left[nd-\frac{d(d-1)}{2}-1-d_{vi}(d-1)\right]x_{j}.
\end{aligned}
\end{equation}
Analogously for the component $x_{j}$ , we have
 \begin{equation*}
\begin{aligned}  
\rho(G)x_{j}&=\sum_{k=1}^{n}d_{jk}[(1-\alpha)x_{k}+\alpha x_{j}]\\& \leq d_{vj}[(1-\alpha)x_{i}+\alpha x_{j}]+2[(1-\alpha)x_{i}+\alpha x_{j}]+3[(1-\alpha)x_{i}+\alpha x_{j}]+\cdots+(d-1)[(1-\alpha)\\&~~~~
x_{i}+\alpha x_{j}]+d[n-1-d_{vj}-(d-2)][(1-\alpha)x_{i}+\alpha x_{j}]\\&
=\left[nd-\frac{d(d-1)}{2}-1-d_{vj}(d-1)\right][(1-\alpha)x_{i}+\alpha x_{j}].
\end{aligned}
\end{equation*}
 i.e.,
\begin{equation}\label{gong11}
\begin{aligned}  
\left\{\rho(G)-\alpha\left[nd-\frac{d(d-1)}{2}-1-d_{vj}(d-1)\right]\right\}x_{j}\leq(1-\alpha)\left[nd-\frac{d(d-1)}{2}-1-d_{vj}(d-1)\right]x_{i}.
\end{aligned}
\end{equation}
 Combining (\ref{gong10}) and (\ref{gong11}) we get $\{\rho(G)-\alpha[nd-\frac{d(d-1)}{2}-1-d_{vi}(d-1)]\}\{\rho(G)-\alpha[nd-\frac{d(d-1)}{2}-1-d_{vj}(d-1)]\}x_{i}x_{j}
 \leq(1-\alpha)^{2}[nd-\frac{d(d-1)}{2}-1-d_{vi}(d-1)][nd-\frac{d(d-1)}{2}-1-d_{vj}(d-1)]x_{i}x_{j}$, which implies that
 \begin{equation*}
\begin{aligned}  
&\rho^{2}(G)-\alpha[2nd-d(d-1)-2-(d-1)(d_{vi}+d_{vj})]\rho(G)+(2\alpha-1)[nd-\frac{d(d-1)}{2}-1-d_{vi}(d-1\\&)][nd-\frac{d(d-1)}{2}-1-d_{vj}(d-1)]\geq0.
\end{aligned}
\end{equation*}
  i.e.,
  \begin{small}
 \begin{equation*}
\begin{aligned}  
\rho(G)&\leq\frac{\alpha[2nd-d(d-1)-2-(d-1)(d_{vi}+d_{vj})]+\sqrt{\alpha^{2}[2nd-d(d-1)-2-(d-1)(d_{vi}+d_{vj})]^{2}
-4(2}}{2}\\&~~~~
\frac{\overline{\alpha-1)[nd-\frac{d(d-1)}{2}-d_{vi}(d-1)-1][nd-\frac{d(d-1)}{2}-d_{vj}(d-1)-1]}}{2}\\&
=\frac{\alpha[2nd-d(d-1)-2-(d-1)(\delta_{1}+\delta_{2})]+\sqrt{\alpha^{2}[2nd-d(d-1)-2-(d-1)(\delta_{1}+\delta_{2})]^{2}
-4(2\alpha-}}{2}\\&~~~~
\frac{\overline{1)[nd-\frac{d(d-1)}{2}-\delta_{1}(d-1)-1][nd-\frac{d(d-1)}{2}-\delta_{2}(d-1)-1]}}{2}.
\end{aligned}
\end{equation*}
\end{small}
 Suppose that $G$ is a regular graph and the diameter of $G$ is at most $2$. Thus, all
the components $x_{i}$ are equal. If $d(G) = 1$, $G\cong K_{n}$,
and $\rho(G)=n-1$. Thus, equality in (\ref{gong9}) holds. If $d(G)=2$,
we get $\rho(G)x_{i}=(2n-2-d_{vi})x_{i}$. Thus, $\rho(G)=2n-2-d_{vi}$, and the equality in (\ref{gong9}) holds.
\qed\end{proof}

We now turn our attention to obtain bounds of $\rho(G)$ for graphs which are not transmission regular.
\begin{thm}\label{thm48}
 Let $G$ be a connected graph of order $n$, where $n\geq2$.
If $Tr_{1}\geq\cdots\geq Tr_{n}$ and $Tr_{1}>Tr_{n-k+1}$, where $1\leq k\leq n-1$, then
\begin{equation}\label{gong12}
\begin{aligned}  
\rho(G)\leq&\frac{\alpha(Tr_{n-k+1}+1)+Tr_{1}-1+\sqrt{\left[Tr_{1}-\alpha Tr_{n-k+1}-(1-\alpha)(2k-1)\right]^{2}-4k(1-\alpha)^{2}(k-}}{2}\\&
\frac{\overline{Tr_{n-k+1}-1)}}{2},
\end{aligned}
\end{equation}
where $0\leq\alpha<1$, with equality holding if and only if $G$ is a graph with $k$ $(k\leq n-2)$ vertices of degree $n-1$ and the remaining $n-k$ vertices have equal degree less than $n-1$.
\end{thm}
\begin{proof}
Let $V_{1}=\{v_{1},v_{2},\cdots,v_{n-k}\}$ and $V_{2}=V(G)\backslash V_{1}$. Then $\mathcal{D}_{\alpha}(G)$ may be partitioned as
$$
\mathcal{D}_{\alpha}(G)=(1-\alpha)
\left[
\begin{array}{cc}
\mathcal{D}_{11} & \mathcal{D}_{12}\\
\mathcal{D}_{21} & \mathcal{D}_{22}\\
\end{array}
\right]
+\alpha
\left[
\begin{array}{cc}
Tr_{11} & 0\\
0 & Tr_{22}\\
\end{array}
\right],
$$
where $\mathcal{D}_{11}$ and $Tr_{11}$ are $(n-k)\times(n-k)$ matrices. Let
$$
U=
\left[
\begin{array}{cc}
\frac{1}{x}I_{n-k} & 0\\
0 & I_{k}\\
\end{array}
\right].
$$
for $0<x<1$ (to be determined) and $B=U^{-1}\mathcal{D}_{\alpha}(G)U$, where $I_{s}$ is the $s\times s$ unit matrix. Then
$$
B=(1-\alpha)
\left[
\begin{array}{cc}
\mathcal{D}_{11} & x\mathcal{D}_{12}\\
\frac{1}{x}\mathcal{D}_{21} & \mathcal{D}_{22}\\
\end{array}
\right]
+\alpha
\left[
\begin{array}{cc}
Tr_{11} & 0\\
0 & Tr_{22}\\
\end{array}
\right].
$$
is a non-negative irreducible matrix that has the same spectrum as $\mathcal{D}_{\alpha}(G)$ for $0\leq\alpha<1$. Let
$B_{i}$ denote the $i$th row sum of $B$. If $i = 1,2,\cdots,n-k$, then since $d_{ij}\geq1$, for
$j = n-k+1,\cdots,n$, we have
 \begin{equation*}
\begin{aligned}  
B_{i}&=(1-\alpha)\left(\sum_{j=1}^{n-k}d_{ij}+x\sum_{j=n-k+1}^{n}d_{ij}\right)+\alpha\sum_{j=1}^{n}d_{ij}\\&
=(1-\alpha)\left(\sum_{j=1}^{n}d_{ij}+(x-1)\sum_{j=n-k+1}^{n}d_{ij}\right)+\alpha\sum_{j=1}^{n}d_{ij}\\&
=Tr_{i}+(1-\alpha)(x-1)\sum_{j=n-k+1}^{n}d_{ij}\\&
\leq Tr_{i}+(1-\alpha)(x-1)k\leq Tr_{1}+(1-\alpha)(x-1)k.
\end{aligned}
\end{equation*}
 If $i = n-k+1,\cdots,n$, then since $d_{ii}=0$ and $d_{ij}\geq1$, for
$j = n-k+1,\cdots,n$ with $i\neq j$, we have
 \begin{equation*}
\begin{aligned}  
B_{i}&=(1-\alpha)\left(\frac{1}{x}\sum_{j=1}^{n-k}d_{ij}+\sum_{j=n-k+1}^{n}d_{ij}\right)+\alpha\sum_{j=1}^{n}d_{ij}\\&
=\left((1-\alpha)\frac{1}{x}+\alpha\right)\sum_{j=1}^{n}d_{ij}+(1-\alpha)\left(1-\frac{1}{x}\right)\sum_{j=n-k+1}^{n}d_{ij}\\&
\leq \left((1-\alpha)\frac{1}{x}+\alpha\right)Tr_{i}+(1-\alpha)\left(1-\frac{1}{x}\right)(k-1)\\&
\leq \left((1-\alpha)\frac{1}{x}+\alpha\right)Tr_{n-k-1}+(1-\alpha)\left(1-\frac{1}{x}\right)(k-1).
\end{aligned}
\end{equation*}
Let
$$Tr_{1}+(1-\alpha)(x-1)k=\left((1-\alpha)\frac{1}{x}+\alpha\right)Tr_{n-k-1}+(1-\alpha)\left(1-\frac{1}{x}\right)(k-1).$$
Then
\begin{equation*}
\begin{aligned}
x=&\frac{(1-\alpha)(2k-1)+\alpha Tr_{n-k+1}-Tr_{1}+\sqrt{\left[Tr_{1}-\alpha Tr_{n-k+1}-(1-\alpha)(2k-1)\right]^{2}-4k(1-\alpha)^{2}}}{2k(1-\alpha)}\\&
\frac{\overline{(k-Tr_{n-k+1}-1)}}{2k(1-\alpha)},
\end{aligned}
\end{equation*}
\begin{equation*}
\begin{aligned}
&Tr_{1}+(1-\alpha)(x-1)k\\&=\frac{\alpha(Tr_{n-k+1}+1)+Tr_{1}-1+\sqrt{\left[Tr_{1}-\alpha Tr_{n-k+1}-(1-\alpha)(2k-1)\right]^{2}-4k(1-\alpha)^{2}(k-}}{2}\\&
~~~~\frac{\overline{Tr_{n-k+1}-1)}}{2}.
\end{aligned}
\end{equation*}
Since $Tr_{1} > Tr_{n-k+1}\geq Tr_{n} \geq n -1 > k - 1$, we have $0 < x < 1$.
Thus, by Lemma \ref{thm12},
\begin{equation*}
\begin{aligned}
\rho(G)&\leq\max_{1\leq i\leq n}B_{i}\\&
\leq\frac{\alpha(Tr_{n-k+1}+1)+Tr_{1}-1+\sqrt{\left[Tr_{1}-\alpha Tr_{n-k+1}-(1-\alpha)(2k-1)\right]^{2}-4k(1-\alpha)^{2}(k-}}{2}\\&~~~~
\frac{\overline{Tr_{n-k+1}-1)}}{2}.
\end{aligned}
\end{equation*}
Suppose that the equality holds in (\ref{gong12}). Since $B_{i}=Tr_{1}+(1-\alpha)(x-1)k$, for
$i = 1,2,\cdots,n-k$, we have $d_{ij} = 1$, for $i = 1,2,\cdots,n-k$ and $j = n -k +1,\cdots,n$,
which implies that every vertex in $V_{1}$ is adjacent to all vertices in $V_{2}$. Again, since
$B _{i} =\left((1-\alpha)\frac{1}{x}+\alpha\right)Tr_{n-k-1}+(1-\alpha)\left(1-\frac{1}{x}\right)(k-1)$, for $i=n-k +1,\cdots,n$, we have $d _{ij}
= 1$, for $i,j = n-k+1,\cdots,n$ with $i\neq j$, which implies that $V _{2}$ induces a complete subgraph
in $G$. Thus, the degree of every vertex in $V_{2}$ is $n-1$ and hence the diameter of $G$
is at most $2$. Since $Tr _{1} = Tr _{2 }= \cdots = Tr _{n-k}$ , every vertex in $V _{1}$ has the same
degree. Moreover, since $Tr _{1} > Tr _{n-k+1}$, $G$ cannot be the complete graph, and thus, $k \leq n-2$.

 Conversely, if $G$ is a graph stated in the second part of the Theorem, then from
the proof above, we have $B _{1} = B _{2} = \cdots = B _{n}$ and thus, the equality holds.
\qed\end{proof}

Recall that, the line graph $L(G)$ of a graph $G$ is a graph such that the vertices of
$L(G)$ are the edges of $G$ and two vertices of $L(G)$ are adjacent if and only if their
corresponding edges in $G$ share a common vertex \cite{12}.

\begin{thm}\label{thm49}
 Let $G$ be a connected graph with $n$ vertices, $m$ edges and $d_{i}=deg(v_{i})$.
If $diam(G)\leq2$ and $G$ does not contain $F_{i}$, $i=1,2,3$ as an induced subgraph, then
\begin{equation*}
\begin{aligned}  
\rho(L(G))\geq\frac{2m^{2}-\sum_{i=1}^{n}d_{i}^{2}}{m}.
\end{aligned}
\end{equation*}
\end{thm}
\begin{proof}
Let $G$ be a connected graph of diameter $2$, which does not contain $F_{i}$ for $i=1,2,3$ as an induced subgraph, and let its vertices be labeled as $v_{1}, v_{2}, \cdots, v_{n}$. Let $d_{i}$ denote the degree of $v_{i}$. Then, as $G$ is of diameter $2$, it is easy to observe that the
$i$th row of $\mathcal{D}_{\alpha}(G)$ consists of $(1-\alpha)d_{i}$ one's, $(1-\alpha)(n-d_{i}-1)$ two's and diagonal entry $\alpha(2n-d_{i}-2)$.
Let $X=(1,1,\cdots,1)^{T}$ be the all one vector. Then by the Raleigh's principle,
\begin{equation*}
\begin{aligned}  
\rho(G)\geq\frac{X^{T}\mathcal{D}_{\alpha}(G)X}{X^{T}X}=\frac{1}{n}\sum_{i=1}^{n}(2n-d_{i}-2)=\frac{2n^{2}-2n-2m}{n}.
\end{aligned}
\end{equation*}
The number of vertices of $L(G)$ is $n_{1}=m$ and the number of edges of $L(G)$ is $m_{1}=\frac{1}{2}\sum_{i=1}^{n}d_{i}^{2}-m$. Now since $G$ has no $F_{i}$
for $i=1,2,3$ as its induced subgraph, by Lemma \ref{lem13}, we get $diam(L(G))\leq2$. Therefore,
\begin{equation*}
\begin{aligned}  
\rho(L(G))&\geq\frac{2n_{1}^{2}-2n_{1}-2m_{1}}{n_{1}}\\&=\frac{2m^{2}-2m-2(\frac{1}{2}\sum_{i=1}^{n}d_{i}^{2}-m)}{m}\\&=\frac{2m^{2}-\sum_{i=1}^{n}d_{i}^{2}}{m}.
\end{aligned}
\end{equation*}
\qed\end{proof}

\begin{cor}\label{thm410}
 If $G$ is a connected $r$-regular graph on $n$ vertices and none of $F_{i}$, $i=1,2,3$ is an induced subgraph of $G$, then
\begin{equation*}
\begin{aligned}  
\rho(L(G))\geq r(n-2).
\end{aligned}
\end{equation*}
\end{cor}
\begin{proof}
 Since $G$ is an $r$-regular graph on $n$ vertices, the number of edges of $G$ is $m=\frac{nr}{2}$
and $d_{i}=deg(v_{i})=r$. Then from Theorem \ref{thm49}, we get $\rho(L(G))\geq r(n-2)$.
\qed\end{proof}
\begin{thm}\label{thm411}
 Let G be a connected graph with vertex set $V(G)=\{v_{1},v_{2},\cdots,v_{n}\}$ and edge set $E(G)=\{e_{1},e_{2},\cdots,e_{m}\}$. Let $deg(e_{i})$ denote the number of edges adjacent to $e_{i}$. Then
\begin{equation*}
\begin{aligned}  
\rho(L(G))\geq2(m-1)-\frac{1}{m}\sum_{i=1}^{m}deg(e_{i}).
\end{aligned}
\end{equation*}
\end{thm}
\begin{proof}
Consider an edge $e=uv$ which is adjacent to $deg(u)+deg(v)-2=deg(e)$ edges at $u$ and $v$ taken together.
Hence the edge $e$ is not adjacent to remaining $m-1-deg(e)$ edges of $G$. In $L(G)$ the
distance between $e$ and the remaining these $m-1-deg(e)$ vertices is more than $1$.
Hence each edge $e=uv$ contributes the distance at least $2(m-1-deg(e))$ in $L(G)$.
Let $X=(1,1,\cdots,1)^{T}$ be the all one vector of size $m$. Then by the Raleigh's principle,
\begin{equation*}
\begin{aligned}  
\rho(L(G))\geq\frac{X^{T}\mathcal{D}_{\alpha}(L(G))X}{X^{T}X}\geq\frac{1}{m}\sum_{i=1}^{m}(2m-deg{e_{i}}-2)=2(m-1)-\frac{1}{m}\sum_{i=1}^{m}deg(e_{i}).
\end{aligned}
\end{equation*}
\qed\end{proof}

\section{The generalized distance spectrum of some composite graphs}
\begin{thm}\label{c42}
Let $G$ be a $k-$transmission regular graph of order $p$ with its distance spectrum $\{\mu_{1}^{\mathcal{D}},\mu_{2}^{\mathcal{D}},\cdots,\mu_{p}^{\mathcal{D}}\}$. Let $H$ be an $r-$regular graph on $n$ vertices with its adjacency spectrum $\{r,\lambda_{1},\lambda_{2},\cdots,\lambda_{n}\}$. Then, the generalized distance spectrum of $G[H]$ is:\\
$(1)$ $\alpha(kn+2n-r-2)+(1-\alpha)(n\mu_{i}^{\mathcal{D}}+2n-r-2)$ for $i=1,2,\cdots,p$;\\
$(2)$ $\alpha(kn+2n-r-2)-(1-\alpha)(2+\lambda_{j})$, $p$ times, $j=2,3,\cdots,n$.
\end{thm}
\begin{proof} Let $G$ and $H$ be two connected graphs with at least two vertices and let $u=(u_{1},v_{1}), v=(u_{2},v_{2})\in V(G)\times V(H)$. Then
\begin{equation*}
    d_{G[H]}(u,v)=
   \begin{cases}
   d_{G}(u_{1},u_{2})&\mbox{if $u_{1}\neq u_{2}$}\\
   1&\mbox{if $u_{1}= u_{2}$ and $v_{1}$ adjacent to $v_{2}$}\\
    2&\mbox{if $u_{1}= u_{2}$ and $v_{1}$ not adjacent to $v_{2}$}\\
   \end{cases}
  \end{equation*}
By a proper ordering of the vertices of $G[H]$,
its distance matrix $\mathcal{D}(G[H])$ can be written in the form
\begin{equation*}
\mathcal{D}(G[H])=\mathcal{D}(G)\otimes J_{n}+I_{p}\otimes(2(J-I)-A(H)).
\end{equation*}

For graph $G[H]$, the transmission of every vertex is $Tr(u)=kn+2n-r-2$. Then, its transmission matrix $Tr(G[H])$ can be written in the form
\begin{equation*}
Tr(G[H])=I_{p}\otimes (kn+2n-r-2)I_{n}.
\end{equation*}
Thus, the generalized distance matrix $\mathcal{D_{\alpha}}(G[H])$ can be written as
\begin{equation*}
\mathcal{D_{\alpha}}(G[H])=\alpha I_{p}\otimes (kn+2n-r-2)I_{n}+(1-\alpha)[\mathcal{D}(G)\otimes J_{n}+I_{p}\otimes(2(J-I)-A(H))],
\end{equation*}
where $J$ is the all-one matrix, and $I$ is the identity matrix of appropriate orders.

As a regular graph, $H$ has the all-one vector $\mathbf{1}$ as an eigenvector corresponding to the eigenvalue $r$ , while all the
other eigenvectors are orthogonal to $\mathbf{1}$.
Let $\lambda_{j}\neq r$, $j=2,3,\cdots,n$ be an eigenvalue of $A(H)$ with an eigenvector $Y_{j}$, such that
$\mathbf{1}^{T} Y_{j} = 0$ and $A(H)Y_{j}=\lambda_{j}Y_{j}$. Thus, we have
$(2(J-I)-A(H))\mathbf{1}=(2n-r-2)\mathbf{1}$ and
 $(2(J-I)-A(H))Y_{j}=-(\lambda_{i}+2)Y_{j}$.\\
Let $X_{i}$, $i=1,2,\cdots,p$ be an eigenvector corresponding to the eigenvalue $\mu_{i}^{\mathcal{D}}$ of $\mathcal{D}(G)$. Therefore
\begin{equation*}
 \mathcal{D}(G)X_{i}=\mu_{i}^{\mathcal{D}}X_{i}
\end{equation*}

Now
$$\scriptsize{
\begin{aligned} 
\mathcal{D_{\alpha}}(G[H])(X_{i}\otimes \mathbf{1_{n}})&=
\{\alpha I_{p}\otimes (kn+2n-r-2)I_{n}+(1-\alpha)[\mathcal{D}(G)\otimes J_{n}+I_{p}\otimes(2(J-I)-A(H))]\}(X_{i}\otimes \mathbf{1_{n}})\\  &
=\alpha I_{p}X_{i}\otimes(kn+2n-r-2)I_{n}\mathbf{1_{n}}+(1-\alpha)\mathcal{D}(G)X_{i}\otimes J_{n}\mathbf{1_{n}}+(1-\alpha)I_{p}X_{i}\otimes(2(J-I)-A(H))\mathbf{1_{n}}\\&
=\alpha X_{i}\otimes(kn+2n-r-2)\mathbf{1_{n}}+(1-\alpha)\mu_{i}^{\mathcal{D}}X_{i}\otimes n\mathbf{1_{n}}+(1-\alpha)X_{i}\otimes(2n-r-2)\mathbf{1_{n}}\\&
=\alpha(kn+2n-r-2)X_{i}\otimes\mathbf{1_{n}}+(1-\alpha)n\mu_{i}^{\mathcal{D}}X_{i}\otimes \mathbf{1_{n}}+(1-\alpha)(2n-r-2)X_{i}\otimes\mathbf{1_{n}}\\&
=[\alpha(kn+2n-r-2)+(1-\alpha)(n\mu_{i}^{\mathcal{D}}+2n-r-2)]X_{i}\otimes\mathbf{1_{n}}.
\end{aligned}
}$$
Therefore, $\alpha(kn+2n-r-2)+(1-\alpha)(n\mu_{i}^{\mathcal{D}}+2n-r-2)$, $i=1,2,\cdots,p$ is an eigenvalue of $\mathcal{D_{\alpha}}(G[H])$ with eigenvector $X_{i}\otimes \mathbf{1_{n}}$.

Let $\{Z_{k}\}$, $k=1,2,\cdots,p$ be the family of $p$ linearly independent eigenvectors associated with the eigenvalue $1$ of $I_{p}$. Then for each $j=2,3,\cdots,n$, the $p$ vectors $Z_{k}\otimes Y_{j}$ are eigenvectors of $\mathcal{D_{\alpha}}(G[H])$ with eigenvalue $\alpha(kn+2n-r-2)-(1-\alpha)(2+\lambda_{j})$. For
$$\scriptsize{
\begin{aligned} 
\mathcal{D_{\alpha}}(G[H])(Z_{k}\otimes Y_{j})&=
\{\alpha I_{p}\otimes (kn+2n-r-2)I_{n}+(1-\alpha)[\mathcal{D}(G)\otimes J_{n}+I_{p}\otimes(2(J-I)-A(H))]\}(X_{i}\otimes \mathbf{1_{n}})\\  &
=\alpha I_{p}Z_{k}\otimes(kn+2n-r-2)I_{n}Y_{j}+(1-\alpha)\mathcal{D}(G)Z_{k}\otimes J_{n}Y_{j}+(1-\alpha)I_{p}Z_{k}\otimes(2(J-I)-A(H))Y_{j}\\&
=\alpha Z_{k}\otimes(kn+2n-r-2)Y_{j}+(1-\alpha)\mathcal{D}(G)Z_{k}\otimes 0+(1-\alpha)Z_{k}\otimes(-\lambda_{j}-2)Y_{j}\\&
=[\alpha(kn+2n-r-2)-(1-\alpha)(\lambda_{j}+2)](Z_{k}\otimes Y_{j}).
\end{aligned}
}$$
Also, the $pn$ vectors $X_{i}\otimes \mathbf{1_{n}}$ and $Z_{k}\otimes Y_{j}$ are linearly independent. As the eigenvectors belonging to different eigenvalues are linearly independent and as $\mathcal{D_{\alpha}}(G[H])$ has a basis consisting entirely of eigenvectors, the theorem follows. \qed\end{proof}

\begin{thm}\label{c45}
 Let Ham$(3,n)$ be the cubic lattice graph of characteristic $n$. Then, the generalized distance spectrum of cubic lattice graph is:\\
$(1)$ $3n^{2}(n-1)$;\\
$(2)$ $(\alpha-1)n^{2}+3\alpha n^{2}(n-1)$, with multiplicity $3(n-1)$;\\
$(3)$ $3\alpha n^{2}(n-1)$, with multiplicity $n^{3}-3n+2$;\\
\end{thm}
\begin{proof}
The graph $K_{n}$ is distance regular with distance regularity $n-1$. Now the proof follows by repeated application of Lemma \ref{lem14} and from the distance spectrum of $K_{n}$.
\qed\end{proof}
The graph $C_{k}+C_{m}$ where both $k$ and $m$ are odd is defined as the $C_{4}$ nanotori, $T_{k,m,C_{4}}$\cite{23}.

\begin{thm}\label{CoronaTh4}
 The generalized distance spectrum of the $C_{4}$ nanotori, $T_{k,m,C_{4}}$ consists of the following numbers\\
$(1)$ $\frac{(m+k)(mk-1)}{4}$;\\
$(2)$ $-\frac{m(1-\alpha)}{4}\sec^{2}(\frac{\pi j}{2k})+\frac{\alpha(m+k)(mk-1)}{4}$, $j\in\{1,2,\cdots,k-1\}$ and even;\\
$(3)$ $-\frac{m(1-\alpha)}{4}\csc^{2}(\frac{\pi r}{2k})+\frac{\alpha(m+k)(mk-1)}{4}$, $r\in\{1,2,\cdots,k-1\}$ and odd;\\
$(4)$ $-\frac{k(1-\alpha)}{4}\sec^{2}(\frac{\pi t}{2m})+\frac{\alpha(m+k)(mk-1)}{4}$, $t\in\{1,2,\cdots,m-1\}$ and even;\\
$(5)$ $-\frac{k(1-\alpha)}{4}\csc^{2}(\frac{\pi l}{2m})+\frac{\alpha(m+k)(mk-1)}{4}$, $l\in\{1,2,\cdots,m-1\}$ and odd;\\
$(6)$ $\frac{\alpha(m+k)(mk-1)}{4}$, with multiplicity $(k-1)(m-1)$.
\end{thm}
\begin{proof}
The cycle $C_{2n+1}$ is distance regular with distance regularity $n(n+1)$. Now the proof follows from Lemma \ref{lem14} and Lemma \ref{lem15}.
\qed\end{proof}


\end{document}